\documentclass[12pt]{amsart}

\usepackage[utf8]{inputenc}

\usepackage{amssymb, graphicx,color}

\usepackage{amsfonts}
\usepackage{amsmath}
\usepackage{latexsym}
\usepackage{amscd}
\usepackage{verbatim}
\usepackage{mathrsfs}
\usepackage{hyperref}
\usepackage{bbm}

\newcommand{\gl}{\mf{gl}}

\newcommand{\id}{{\rm id}}

\allowdisplaybreaks[1]

\def\span{\operatorname{span}}

\newcommand{\Z}{{\mathbb Z}}
\newcommand{\C}{{\mathbb C}}
\newcommand{\N}{{\mathbb N}}
\newcommand{\Q}{{\mathbb Q}}

\newcommand{\supp}{{\operatorname{Supp}}\xspace}

\renewcommand{\phi}{\varphi}

\newcommand{\Ker}{\operatorname{Ker}}
\def\Ann{\operatorname{Ann}}
\def\Hom{\operatorname{Hom}}
\def\supp{\mathrm{Supp}}
\def\Vir{\mathrm{Vir}}

\def\gl{\mathfrak{gl}}

\def\Rad{\text{Rad}}

\def\span{\text{span}}

\def\C{\mathbb{C}}
\def\Z{\mathbb{Z}}
\def\N{\mathbb{N}}

\def\m{\mathfrak{m}}

\newtheorem{theorem}{Theorem}[section]
\newtheorem{proposition}[theorem]{Proposition}
\newtheorem{lemma}[theorem]{Lemma}
\newtheorem{corollary}[theorem]{Corollary}

\theoremstyle{remark}
\newtheorem{remark}[theorem]{Remark}

\numberwithin{equation}{section}

\title[simple weight modules with finite dimensional weight spaces]{Classification of simple Harish-Chandra modules over the generalized Witt algebras}
\author{Rencai L\"u,  Yaohui Xue}
\date{} 

\begin{document}

\maketitle

\begin{abstract}
In this paper, we classify simple Harish-Chandra modules over simple generalized Witt algebras.\end{abstract}

\vskip 10pt \noindent {\em Keywords:} Witt algebra, generalized Witt algebra, simple module, bounded module, weight module
\vskip 5pt
\noindent
{\em 2010  Math. Subj. Class.:} 17B10, 17B20, 17B65, 17B66, 17B68

\section{Introduction}
Let $\C,\Q,\Z,\Z_+$ and $\N$ be the set of complex numbers, rational numbers, integers, non-negative integers and positive integers, respectively.
Throughout this paper, $G$ is a nonzero abelian group and $D$ is a vector space over $\C$, with a map $\phi:G\times D\rightarrow\C$ that is additive in the first variable and $\C$-linear in the second one. For convenience we shall also use the following notations
$$\phi(a,d)=d(a)=a(d),\forall a\in G, d\in D.$$


Let $A=\C[G]$.
The generalized Witt algebra $W=W(G,D,\phi)=A\otimes_{\C} D$ is defined by the Lie brackets
\begin{equation}[t^a\otimes d, t^{a'}\otimes d']=t^{a+a'}\otimes (\phi(a',d)d'-\phi(a,d')d),\,\forall a,a'\in G, d,d'\in D.\end{equation}
For simplicity we write $t^ad$ instead of $t^a\otimes d$, and write $d$ instead of $t^0\otimes d$.

E. Cartan \cite{C} introduced the four classes of Cartan type Lie algebras, including Witt algebras, in 1909. I. Kaplansky defined generalized Witt algebras over fields of prime characteristic in \cite{K}.  The definition of generalized Witt algebras over fields of characteristic $0$ was given by N. Kawamoto in \cite{K1}. An equivalent definition of generalized Witt algebras was given by D. Djokovic and K. Zhao in \cite{DZ}.

Over the last two decades, the representation theory of generalized Witt algebras was extensively studied by many mathematicians and physicists. Virasoro algebra, higher rank Virasoro algebras, $\Q$-Virasoro algebra and generalized Virasoro algebras are one dimensional central extensions of some simple generalized Witt algebras $W(G,D,\phi)$ with $\dim D=1$. Simple Harish-Chandra modules (i.e. simple weight modules with finite-dimensional weight spaces) over the Virasoro algebra were conjectured by V. Kac in \cite{Kac}. O. Mathieu proved this conjecture in \cite{Mat}, see also \cite{Su} and references therein. These modules for higher rank Virasoro algebras were classified in \cite{LZ2,Su1}. V. Mazorchuk classified the simple Harish-Chandra modules over the $\Q$-Virasoro algebra in \cite{M}.  X. Guo, R. L\"u and K. Zhao classified the simple Harish-Chandra modules over the generalized Virasoro algebras in \cite{GLZ1}.

Y. Billig and V. Futorny completed the classification of simple Harish-Chandra modules for Witt algebras in \cite{BF1}. Extended Witt algebras are also an important class of generalized Witt algebras. X. Guo, G. Liu and K. Zhao started the classification of simple Harish-Chandra modules over extended Witt algebras in \cite{GLZ}.
K. Zhao classified the indecomposable weight modules with $1$ dimensional weight spaces for simple generalized Witt algebras in \cite{Z2}. Simple Harish-Chandra modules over Witt superalgebras were classified in \cite{BFIK,XL}. We refer the readers to \cite{BF2,E1,E2,GLLZ,MZ} and the references therein for more related results.

In this paper, we classify simple Harish-Chandra modules over simple generalized Witt algebras. In Section 2, we collect some notations and results for later use. In Section \ref{tensormodule}, we define the tensor modules over the generalized Witt algebras and determine their simplicity and isomorphism group. In Section \ref{boundedmodule}, we classify simple bounded modules, see Theorem \ref{bounded1} and Theorem \ref{bounded2}. In Section \ref{SecHWT}, we prove that any simple Harish-Chandra module that is not a bounded module is a module of highest weight type, see Theorem \ref{hwtype}. Moreover, we prove that a nontrivial module of highest weight type over $W(G,D,\phi)$ is a Harish-Chandra module if and only $\dim D<\infty$, see Theorem \ref{HC} and Theorem \ref{infweightspace}. Here we also give our main theorem of this paper, see Theorem \ref{main}.

\section{Preliminaries}\label{Secpre}

Through out this paper, by $G'\leqslant G$ we mean that $G'$ is a subgroup of $G$, and by $G'\times D'\leqslant G\times D$ we mean that $G'$ is a subgroup of $G$ and $D'$ is a subspace of $D$. $\phi|_{G\times D}$ is called non-degenerate if $\phi(G,d)=0$ and $\phi(a,D)=0$ imply $d=0$ and $a=0$ respectively.

Suppose $G'\times D'\leqslant G\times D$. Then the generalized Witt algebra $W(\phi|_{G'\times D'}):=W(G', D', \phi|_{G'\times D'})$ is a subalgebra of $W$.
If $D'= \C d$ for some $d\in D$ with $\phi|_{G'\times D'}$ non-degenerate, denote $W(\phi|_{G'\times D'})$ also by $\Vir[G', d]$. $\Vir[G', d]$ is isomorphic to a generalized centerless Virasoro algebra relative to the subgroup $d(G')$ of $\C$.

Suppose $G'\times D'\leqslant G\times D$. Let $\Ker_1\phi|_{G'\times D'}=\{a\in G'\ |\ \phi(a,D')=0\},\Ker_2\phi|_{G'\times D'}=\{d\in D'\ |\ \phi(G',d)=0\}$. Then $\Ker_1\phi|_{G'\times D'}$ is a subgroup of $G'$ and $\Ker_2\phi|_{G'\times D'}$ is a subspace of $D'$.

\begin{lemma}\label{fin*fin}
Suppose $\phi$ is non-degenerate. Let $G_0\times D_0\leqslant G\times D$ with $G_0\cong\Z^m$ and $\dim D_0<\infty$.
Then there is $G_1\times D_1\leqslant G\times D$ with $G_0\leqslant G_1$ and $D_0\subset D_1$, such that $G_1\cong\Z^n$ for some $n\in \N$ and $\dim D_1<\infty$ and $\phi|_{G_1\times D_1}$ is non-degenerate.
\end{lemma}
\begin{proof}
Let $K=\Ker_2\phi|_{G_0\times D}$ and $a_1,\dots,a_m$ be a $\Z$-basis of $G_0$. Then $\dim D/K=\dim D/(\cap_{i=1}^m\Ker_2\phi|_{\Z a_i\times D})\le \sum_{i=1}^m\dim D/\Ker_2\phi|_{\Z a_i\times D}=m$.  Let $K'$ be a complementary subspace of $K$ in $D$ and let $D_1=D_0+K'$. Then $\dim D_1<\infty$ and $\Ker_1\phi|_{G_0\times D_1}=0$.

We will prove this lemma by induction on $k=\dim\Ker_2\phi|_{G_0\times D_1}$. If $k=0$, then $\phi|_{G_0\times D_1}$ is non-degenerate by definition. Now suppose $k>0$. Let $a'\in G$ with $\phi(a',\Ker_2\phi|_{G_0\times D_1})\neq 0$. Clearly, $a'\notin G_0$. Let $G'=G+\Z a'$. From $G_0\leqslant G'$ we have $\Ker_2\phi|_{G'\times D_1}\subset \Ker_2\phi|_{G_0\times D_1}$, and from $\phi(a',\Ker_2\phi|_{G_0\times D_1})\neq 0$ we have $\Ker_2\phi|_{G'\times D_1}\subsetneqq \Ker_2\phi|_{G_0\times D_1}$. So $\dim\Ker_2\phi|_{G'\times D_1}<\dim \Ker_2\phi|_{G_0\times D_1}.$ By the induction hypothesis, there is $G_1\leqslant G$ such that $G_1\cong\Z^n$ and $\phi|_{G_1\times D_1}$ is non-degenerate.
\end{proof}

\begin{lemma}\label{G1D1}
Suppose that $\Ker_1\phi=0$, $\Gamma$ is a simple Harish-Chandra $W$-module, and $I, I_G, I_D$ are finite subsets of $\supp(\Gamma), G, D$ ,respectively.
Then there exists $G_1\times D_1\leqslant G\times D$ with the following properties:
\begin{itemize}
\item[(1)]$G_1$ is free of finite rank and $\dim D_1<\infty$;
\item[(2)]$I_G\subset G_1$, $I_D\subset D_1$ and $\lambda_i-\lambda_j\in G_1$ for any $\lambda_i,\lambda_j\in I$;
\item[(3)]if $M$ is a $W(\phi|_{G_1\times D_1})$-submodule of $\Gamma$ with $M\cap\sum_{\lambda\in I}\Gamma_{\lambda}\neq 0$, then $\sum_{\lambda\in I}\Gamma_{\lambda}\subset M$;
\item[(4)]$\Gamma$ has a simple $W(\phi|_{G_1\times D_1})$-subquotient $\Gamma'$ with $\dim \Gamma_{\lambda}=\dim \Gamma'_{\lambda}$ for any $\lambda\in I$.
\end{itemize}
\end{lemma}
\begin{proof}

For any $\lambda,\mu\in I$, define the linear map $f_{\lambda,\mu}:U(W)_{\mu-\lambda}\rightarrow \Hom_\C(\Gamma_{\lambda},\Gamma_{\mu})$ by $f_{\lambda,\mu}(x)(v)=x\cdot v,\ x\in U(W)_{\mu-\lambda},v\in\Gamma_{\lambda}.$
Then $\dim U(W)_{\mu-\lambda}/\Ker f_{\lambda,\mu}<\infty$. Let $V_{\lambda,\mu}$ be a complementary subspace of $\Ker f_{\lambda,\mu}$ in $U(W)_{\mu-\lambda}$. By the simplicity of $\Gamma$,
$$V_{\lambda,\mu}\cdot v=U(W)_{\mu-\lambda}\cdot v=\Gamma_{\mu},\ \forall 0\ne v\in\Gamma_{\lambda},\lambda,\mu\in I.$$


Let $D_0$ be a finite dimensional subspace of $D$ such that $(\lambda-\mu)(D_0)\neq 0$ for any $\lambda,\mu\in I$ with $\lambda\neq\mu$. Then there exists $G_1\times D_1\leqslant G\times D$ such that $G_1$ is free of finite rank, $\dim D_1<\infty, D_0\subset D_1, I_G\subset G_1,I_D\subset D_1, \sum_{\lambda,\mu\in I}V_{\lambda,\mu}\subset U(W(\phi|_{G_1\times D_1}))$
and $\lambda-\mu\in G_1$ for any $\lambda,\mu\in I$. Clearly, $G_1$ and $D_1$ satisfy (1), (2) and (3).

Let $X_1=U(W(\phi|_{G_1\times D_1}))\cdot\sum_{\lambda\in I}\Gamma_{\lambda}$. Then from (3) any proper $W(\phi|_{G_1\times D_1})$-submodule of $X_1$ intersects with $\sum_{\lambda\in I}\Gamma_{\lambda}$ trivially. Let $X_2$ be the sum of all proper $W(\phi|_{G_1\times D_1})$-submodules of $X_1$. Then $X_1/X_2$ is a simple $W(\phi|_{G_1\times D_1})$-module. Thus, $G_1$ and $D_1$ satisfy (4).
\end{proof}

A $W$-module $V$ is called a weight module if $V=\oplus_{\lambda\in D^*}V_{\lambda}$, where
$$V_{\lambda}=\{v\in V\ |\ d\cdot v=\lambda(d)v,\forall d\in D\}.$$
$\supp(V):=\{\lambda\in D^*\ |\ V_\lambda\neq 0\}$ is called the support set of \emph{weight module} $V$. A weight module $V$ is called a \emph{Harish-Chandra module} if $\dim V_\lambda<\infty$ for all $\lambda\in\supp(V)$. A weight module is called a \emph{bounded module} if $\dim V_\lambda$ are uniformly bounded.

Define the generalized extended Witt algebra $\tilde{W}=W\ltimes A$ by the bracket in $W$ and
$$[t^ad,t^{a'}]=d(a')t^{a+a'},d,d'\in D, a',a\in G.$$
A $\tilde W$-module $V$ is called an $AW$-module if
$$t^a\cdot t^b\cdot v=t^{a+b}\cdot v,t^0\cdot v=v,\forall a,b\in G.$$

\section{Structure of tensor modules}\label{tensormodule}
By Theorem 2.2 in \cite{DZ}, $W$ is a simple Lie algebra if and only if $\phi$ is non-degenerate.
Our main purpose is to deal with the case that $\phi$ is non-degenerate. However, for later use, we only assume that $\Ker_1\phi=0$ in Section \ref{tensormodule} and \ref{boundedmodule}. Then $G$ is identified with a subgroup of $D^*$ via $a\mapsto\phi(a,-)$.

Let $R=R(G,D,\phi)=G\otimes_{\Z}D$ be the associative algebra over $\C$ with the multiplication
$$(a\otimes d)\cdot (a'\otimes d')=\phi(a',d)a\otimes d'$$
and let $L$ be the Lie algebra associated with $R$.

For any $L$-module $V$ and any $\C$-linear map $\sigma\in D^*$, define the action of $\tilde W$ on $\Gamma(V,\sigma):=A\otimes_{\C}V$ by
$$t^ad \cdot (t^b \otimes v)=t^{a+b} \otimes(d(b)+\sigma(d)+a\otimes d)v,$$
$$t^a \cdot (t^b \otimes v)=t^{a+b} \otimes v,\ a,b\in G,d\in D, v\in V.$$
\begin{lemma}\label{Gamma}
$\Gamma(V,\sigma)$ is an $AW$-module.
\end{lemma}
\begin{proof}
It is straightforward to verify that
$$t^a d \cdot t^{a'} d'\cdot (t^b \otimes v)-t^{a'} d' \cdot t^{a} d\cdot (t^b \otimes v)=[t^ad,t^{a'}d']\cdot(t^b\otimes v),$$
$$t^a d\cdot t^{a'}\cdot (t^b\otimes v)-t^{a'}\cdot t^ad\cdot (t^b\otimes v)=[t^ad,t^{a'}]\cdot (t^b\otimes v).$$
So $\Gamma(V,\sigma)$ is an $AW$-module.
\end{proof}
$\Gamma(V,\sigma)$ is called a \emph{tensor module} of $W$.
It is easy to see that $\Gamma(V,\sigma)$ is a weight module over $W$.

Let
$$K_1=\{x\in R\ |\ Rx=0\},K_2=G\otimes_{\Z}\Ker_2\phi.$$
Then $K_1,K_2$ are two-sided ideals of $R$. Fix a subspace $\bar D$ of $D$ such that $D=\bar D\oplus\Ker_2\phi$. Then $\phi|_{G\times\bar{D}}$ is non-degenerate and $R=G\otimes_{\Z}\bar{D}+K_2$. We will naturally identify $\bar{D}^*$ with $(D/\Ker_2\phi)^*\subset D^*$. Then $G\subset \bar{D}^*$. And $\bar{D}^*=\span_{\C}G$ if $\dim \bar{D}<\infty$.

$\bar{D}^*\otimes\bar{D}$ is an associative algebra with
$$(f_1\otimes d_1)\cdot(f_2\otimes d_2)=f_2(d_1)f_1\otimes d_2,f_1,f_2\in\bar{D}^*,d_1,d_2\in\bar{D}.$$
Then any simple module of $\bar D^*\otimes \bar D$ is isomorphic to $\bar D^*$ with
$$(f\otimes d)\cdot g=g(d)f,f, g\in\bar{D}^*,d\in\bar{D}.$$
If $\dim\bar D=r<\infty$, then $\bar D^*\otimes \bar D\cong\mathrm{End}(\bar D^*)$ as associative algebras.
\begin{lemma}\label{Rmodule}
Suppose that $\dim\bar D=r<\infty$. Then
\begin{itemize}
\item[(1)]$R/(K_1+K_2)\cong \bar{D}^*\otimes\bar{D}$. Then $L/(K_1+K_2)\cong\gl_r$;
\item[(2)]all non-trivial simple (left) $R$-modules are isomorphic to $\bar{D}^*$;
\item[(3)]$(K_1+K_2)\cdot V=0$ for any finite dimensional simple $L$-module $V$. Hence $V$ can be naturally regarded as a simple module over the Lie algebra $\mathfrak{gl}_r$.
\end{itemize}
\end{lemma}

\begin{proof}(1) Define the associative algebra homomorphism $\theta:R\rightarrow\bar{D}^*\otimes\bar{D}$ by
$$\theta(a\otimes(d'+d))=a\otimes d',a\in G,d'\in\bar{D},d\in\Ker_2\phi.$$
Since $\span_\C G=\bar D^*$, $\theta$ is surjective.
To prove (1), it suffices to prove that $K_1+K_2=\Ker\theta$.

From $RK_1=0$ we have $0=\theta(RK_1)=\theta(R)\theta(K_1)=(\bar D\otimes \bar D^*)\theta(K_1)=\theta(K_1)$. So $K_1\subset\Ker\theta$. By definition, $K_2\subset\Ker\theta$. Then $K_1+K_2\subset\Ker\theta$.

We will then prove that $\Ker\theta\subset K_1+K_2$.
Let
$$\sum_i\sum_ja_{i,j}\otimes c_{i,j}d_j\in G\otimes_\Z\bar D\cap\Ker\theta,$$
where $a_{i,j}\in G,c_{i,j}\in\C$ and $d_j$ are $\C$-linearly independent vectors in $\bar D$. Then $\sum_ic_{i,j}a_{i,j}=0\in\bar D^*$ for all $j$. For any $a\otimes d\in R$,
$$(a\otimes d)\cdot(\sum_ia_{i,j}\otimes c_{i,j}d_j)=\sum_ic_{i,j}a_{i,j}(d)a\otimes d_j=0.$$
Then $\sum_i\sum_ja_{i,j}\otimes c_{i,j}d_j\in K_1$. Note that $R=G\otimes_\Z\bar D+K_2$. Then $\Ker\theta\subset K_1+K_2$. Thus, $K_1+K_2=\Ker\theta$.

(2) Let $V$ be a simple $R$-module. Then $K_i\cdot V=0$ or $V$, $i=1,2$. Note that $K_i\cdot K_i\cdot V=(K_i\cdot K_i)\cdot V=0,i=1,2$. So $K_1\cdot V=K_2\cdot V=0$. Thus, $V$ is naturally a nontrivial simple $R/(K_1+K_2)$-module. By (1), $R/(K_1+K_2)\cong\bar{D}^*\otimes\bar{D}$. So $V$ is isomorphic to $\bar{D}^*$.

(3) Let $\Ann(V)=\{x\in L\ |\ x\cdot V=0\}$. Then $V$ is naturally a simple $L/\Ann(V)$-module.
So $L/\Ann(V)$ is reductive, which means that $\Rad(L/\Ann(V))=Z(L/\Ann(V))$. Note that
$$[K_1+K_2,K_1+K_2]=K_1K_2,\ [K_1K_2,K_1K_2]=0.$$
Then $K_1+K_2\subset\Rad(L)$. Therefore, $[L,K_1+K_2]\subset\Ann(V)$. So we only need to prove that $K_1+K_2\subset[L,K_1+K_2]$.

Let $a\otimes d\in K_2\setminus\{0\}$. Then there exists $d'\in\bar D$ such that $a(d')\neq 0$. Therefore, $[a\otimes d',a\otimes d]=a(d')a\otimes d$. So $K_2\subset[L,K_2]$.

Let $d_1,\dots,d_r$ be a basis of $\bar D$ and $\Sigma_i\Sigma_ja_{i,j}\otimes c_{i,j}d_j\in K_1$.
Then
$$0=\theta(\Sigma_i\Sigma_ja_{i,j}\otimes c_{i,j}d_j)=\Sigma_i\Sigma_jc_{i,j}a_{i,j}\otimes d_j.$$
by (1). So $\Sigma_ic_{i,j}a_{i,j}=0\in\bar D^*$ and $\Sigma_ia_{i,j}\otimes c_{i,j}d_j\in K_1$ for all $j\in\{1,\dots,r\}$. Let $a_j\in G$ with $a_j(d_j)\neq 0$, then
$$[\Sigma_ia_{i,j}\otimes c_{i,j}d_j,(a_j\otimes d_j)/a_j(d_j)]=\Sigma_ia_{i,j}\otimes c_{i,j}d_j.$$
So $\Sigma_i\Sigma_ja_{i,j}\otimes c_{i,j}d_j\in[L,K_1]$. Thus, $K_1\subset[L,K_1+K_2]$.

\end{proof}

\begin{remark}\label{rmkglr}
The isomorphism map in Lemma \ref{Rmodule} (1) can be determined in the following way. Since $\bar{D}^*=\span_{\C}G$, $G$ contains $\C$-linear independent elements $a_1,\dots,a_r$. Let $d_1,\dots,d_r$ be the dual basis of $a_1,\dots,a_r$ in $\bar D$. For any $a\in G$ and $d\in D$, $a$ can be uniquely written as $a=k_1a_1+\dots+k_ra_r$ for some $k_1,\dots,k_r\in\C$, and $d$ can be uniquely written as $d=l_1d_1+\dots+l_rd_r+d'$ for some $l_1,\dots,l_r\in\C$ and some $d'\in\Ker_2\phi$. Then The isomorphism map in Lemma \ref{Rmodule} (1) can be given by $a_i\otimes d_j\mapsto E_{i,j},i,j=1,\dots,r$. Note that $W$ has a subalgebra $W(\phi|_{G_0\times\bar D})$ that is isomorphic to the Witt algebra $W_r$, where $G_0=\langle a_1,\dots,a_r\rangle_\Z$ and $W_r=\mathrm{Der}(\C[t_1^{\pm 1},\dots,t_r^{\pm 1}])$.
\end{remark}


\begin{lemma}\label{simpleG}
Suppose $V$ is an $L$-module and $\sigma\in D^*$. Then $\Gamma(V,\sigma)$ is a simple $AW$-module if and only if $V$ is a simple $L$-module.
\end{lemma}
\begin{proof}
By Lemma \ref{Gamma}, for any $L$-submodule $V'$ of $V$, $\Gamma(V',\sigma)=A\otimes V'$ is an $AW$-submodule of $\Gamma(V,\sigma)=A\otimes V$. Then $V$ is a simple $L$-module if $\Gamma(V,\sigma)$ is a simple $AW$-module.

Now suppose $V$ is a simple $L$-module and $\Gamma'$ is a nonzero $AW$-submodule of $\Gamma(V,\sigma)$.
For any nonzero $t^{a_i}\otimes v_i\in\Gamma'$ and any $a\in G,d\in D$ we have
$$t^{-a}\cdot t^ad\cdot t^{a_i}\otimes v_i-(d(a_i)+\sigma(d))t^{a_i}\otimes v_i=t^{a_i}\otimes(a\otimes d)v_i\in\Gamma'.$$
Then $t^{a_i}\otimes U(L)v_i=t^{a_i}\otimes V\subset\Gamma'$. It follows that $A\cdot t^{a_i}\otimes V=A\otimes V\subset\Gamma'$. Thus, $\Gamma(V,\sigma)$ is a simple $AW$-module.
\end{proof}

If $\dim\bar D=r<\infty$ and $\sigma(\Ker_2\varphi)=0$, define the linear map
$$\pi_k:\Gamma(\wedge^k\bar D^*,\sigma)\rightarrow\Gamma(\wedge^{k+1}\bar D^*,\sigma)$$
for $k=0,1,\dots,r-1$ by
$$\pi_k(t^a\otimes f_1\wedge\dots\wedge f_k)=t^a\otimes(a+\sigma)\wedge f_1\wedge\dots\wedge f_k,\ a\in G,f_1,\dots,f_k\in\bar D^*,$$
where $f\wedge c=cf$ for any $c\in\C=\wedge^0\bar D^*$ and $f\in\bar D^*$. Then $\pi_{k+1}\circ\pi_k=0$ for $k=0,\dots,r-2$.
It is straightforward to verify that $\pi_k$ is a $W$-module homomorphism.
Let
\begin{equation}
\Gamma(\sigma,k)=\pi_{k-1}(\Gamma(\wedge^{k-1}\bar D^*,\sigma)),k=1,\dots,r
\end{equation}
and set $\Gamma(\sigma,0)=0$. 

The simplicity of the tensor module $\Gamma(V,\sigma)$ over $W_r$ was determined in \cite{E1}, where $V$ is any finite dimensional simple module.
\begin{lemma}
Suppose $\dim\bar D=r<\infty$ and $\sigma(\Ker_2\varphi)=0$. Then $\Gamma(\sigma,k)$ is a simple $W$-module for $k=1,\dots,r$.
\end{lemma}
\begin{proof}
It suffices to prove that $t^b\otimes v'\in U(W)\cdot t^a\otimes v$ for any $t^a\otimes v,t^b\otimes v'\in\Gamma(\sigma,k)$ with $v\neq 0$. Let $G'\leqslant G$ with $G'\cong\Z^r,b-a\in G'$ and $\phi|_{G'\times\bar D}$ non-degenerate. Then $W(\phi|_{G'\times\bar D})\cong W_r$ and $(\oplus_{g\in G'}\C t^{a+g})\otimes\wedge^k\bar D^*\cong\Gamma_{W(\phi|_{G'\times\bar D})}(\wedge^k\bar D^*,\sigma+a)$ as $W(\phi|_{G'\times\bar D})$-modules. By Theorem 5.5 in \cite{E1}, $\Gamma_{W(\phi|_{G'\times\bar D})}(\sigma+a,k)$ is a simple $W(\phi|_{G'\times\bar D})$-module. Then
$$t^b\otimes v'\in U(W(\phi|_{G'\times\bar D}))\cdot t^a\otimes v \subset U(W)\cdot t^a\otimes v.$$
\end{proof}
Let
$$\tilde\Gamma(\sigma,k)=\Ker\pi_k,k=0,\dots,r-1.$$
Then $\tilde\Gamma(\sigma,k)=\Gamma(\sigma,k)$ if $\sigma\notin G$ and $\tilde\Gamma(\sigma,k)=\Gamma(\sigma,k)+t^{-\sigma}\otimes \wedge^k\bar D^*$ if $\sigma\in G$. Note that $$\Gamma(\wedge^k\bar D^*,\sigma)/\tilde\Gamma(\sigma,k)\cong\Gamma(\sigma,k+1)$$
as $W$-modules for $k=0,\dots,r-1$. Then $\Gamma(\wedge^k\bar D^*,\sigma)$ is not simple for $k=1,\dots,r=1$ and $\Gamma(\wedge^0\bar D^*,\sigma)$ is not simple if $\sigma\in G$. Moreover,
$$\Gamma(\sigma,r)=\span\{t^b\otimes \wedge^r\bar D^*\ |\ b\in G\setminus\{-\sigma\}\}\neq\Gamma(\wedge^r\bar D^*,\sigma)$$
if $\sigma\in G$.

\begin{theorem}\label{simpleGamma}
Suppose $\dim\bar{D}=r<\infty,\sigma\in D^*$, and $V$ is a finite dimensional simple $L$-module. Then $\Gamma(V,\sigma)$ is not simple as $W$-module if and only if one of the following holds:
\begin{itemize}
\item[(1)]$V\cong\wedge^i\bar{D}^*,i\in\{1,...,r-1\}$ and $\sigma(\Ker_2\phi)=0$;
\item[(2)]$V\cong\wedge^i\bar{D}^*,i\in\{0,r\}$, and $\sigma\in G$, where $\wedge^0\bar D=\C$ is the one dimensional trivial $L$-module.
\end{itemize}
\end{theorem}

\begin{proof}
(i) Suppose neither (1) nor (2) is true.

Suppose $\sigma(\Ker_2\phi)\neq 0$. Then there is $d\in\Ker_2\phi$ with $\sigma(d)=1$. By Lemma \ref{Rmodule} (3), $(G\otimes_\Z\C d)\cdot V=0$. For any $t^a\otimes v\in\Gamma(V,\sigma)$ and any $b\in G$, we have
$$t^bd\cdot(t^a\otimes v)=t^{a+b}\otimes((a+\sigma)(d)+b\otimes d)v=t^{a+b}\otimes v=t^b\cdot(t^a\otimes v).$$
So $U(A)(t^a\otimes v)\subset U(W)(t^a\otimes v)$. By Lemma \ref{simpleG}, $\Gamma(V,\sigma)$ is a simple $AW$-module. So
$$\Gamma(V,\sigma)=U(\tilde W)(t^a\otimes v)=U(W)U(A)(t^a\otimes v)=U(W)(t^a\otimes v).$$
Then $\Gamma(V,\sigma)$ is a simple $W$-module.
Now suppose $\sigma(\Ker_2\phi)=0$. Then $V\ncong\wedge^i\bar D^*,i=1,\dots,r-1$.

Let $t^{a_{0}}\otimes v\in\Gamma(V,\sigma)\setminus\{0\},b\in G$.
To prove that $\Gamma(V,\sigma)$ is a simple $W$-module, it suffices to prove that $t^{b}\otimes V\subset U(W)\cdot(t^{a_{0}}\otimes v)$.

Let $\{b_1,\dots,b_r\}$ be a maximal $\C$-linearly independent system in $G$ with $a^0,b\in\langle b_1,...,b_r\rangle$, and let $G_0=\langle b_1,...,b_r\rangle$.
Then $L=L_0+K_1+K_2$, where $L_0=G_0\otimes_{\Z}\bar{D}$. It is easy to see that $W(\phi|_{G_0\times\bar D})$ is isomorphic to $W_r$. By Lemma \ref{Rmodule} (3), $(K_1+K_2)\cdot V=0$. So $V$ is a simple $L_0$-module. By Theorem 1.9 in \cite{E1}, $\C[G_0]\otimes V$ is a simple $W(\phi|_{G_0\times\bar{D}})$-module. So $t^b\otimes V\subset U(W(\phi|_{G_0\times\bar{D}}))\cdot(t^{a_0}\otimes v)\subset U(W)\cdot(t^{a_0}\otimes v)$. Thus, the necessity is proved.

(ii) Suppose (1) or (2) is true. Then $\sigma(\Ker_2\phi)=0$.

By the discussion above, for $k=0,\dots,r-1$, 
%
$\Gamma(\wedge^k\bar D^*,\sigma)$ is not simple if $\sigma\in G$. Thus, the sufficiency is proved.
\end{proof}

By the discussion before Theorem \ref{simpleGamma} we have
\begin{corollary}\label{simplequotient}
Suppose $\dim\bar{D}=r<\infty$, $\sigma\in D^*$, $V$ is a finite dimensional simple $L$-module and $M$ is a simple $W$-quotient of $\Gamma(V,\sigma)$. Then one of the following cases happens:
\begin{itemize}
\item[(1)]$M\cong\Gamma(V,\sigma)$ with $V\ncong\wedge^i\bar D^*,i=0,\dots,r$.
\item[(2)]$M\cong\Gamma(V,\sigma)$ with $\sigma(\Ker_2\phi)\neq 0$.
\item[(3)]$M\cong\Gamma(\sigma,k),k\in\{1,\dots,r\}$ with $\sigma(\Ker_2\phi)=0$.
\item[(4)]$M$ is the one dimensional trivial $W$-module.
\end{itemize}
\end{corollary}

From Lemma \ref{Rmodule} and Corollary \ref{simplequotient} we have
\begin{corollary}\label{cussamedim}
Suppose $\dim\bar{D}=r<\infty$, $V$ is a finite dimensional simple $L$-module, $\sigma\in D^*$, and $M$ is a nontrivial simple $W$-quotient of $\Gamma(V,\sigma)$. Then
\begin{itemize}
\item[(1)]$M$ is bounded with weight spaces of the same dimension, and $\supp(M)=\sigma+G$ or $G\setminus\{0\}$.
\item[(2)]For any $\lambda\in\supp(M)$, if $\dim M_{\lambda}>1$, then $\dim M_{\lambda}\geqslant r$.
\end{itemize}
\end{corollary}
%

\begin{lemma}\label{ncong}
Suppose $\dim \bar{D}= r<\infty$. Let $V, V'$ be two simple finite dimensional $L$-modules and $\sigma,\sigma'\in D^*$.
\begin{itemize}
\item[(1)]If $\sigma(\Ker_2\phi)=0,i,j\in\{1,...,r\}$, then $\Gamma(\sigma,i)\cong\Gamma(\sigma,j)$ if and only if $i=j$.
\item[(2)]$\Gamma(V,\sigma)\cong \Gamma(V',\sigma')$ as $W$-modules if and only if $V\cong V'$ and $\sigma-\sigma'\in G$; or $r=1$, $\sigma-\sigma'\in G,\sigma\notin G$, $V$ (resp. $V'$) $\cong\wedge^0\bar D^*$ and $V'$ (resp. $V$) $\cong\wedge^r\bar D^*$.
\item[(3)]Suppose $\sigma'(\Ker_2\phi)=0$. If $\Gamma(V,\sigma)\cong\Gamma(\sigma',i)$ for some $i\in\{1,\dots,r\}$, then
$$\sigma-\sigma'\in G,\sigma\notin G,V\cong\wedge^0\bar D^*,i=1$$
or
$$\sigma-\sigma'\in G,\sigma\notin G,V\cong\wedge^r\bar D^*,i=r.$$
\end{itemize}
\end{lemma}
\begin{proof}
(1) Suppose there is a $W$-module isomorphism $\psi:\Gamma(\sigma,i)\rightarrow\Gamma(\sigma,j)$ for some $i,j\in\{1,...,r\}$ such that $i<j$. Let $t^b\otimes(b+\sigma)\wedge v\in\Gamma(\sigma,i)\setminus\{0\}$ with $v\in\wedge^{i-1}\bar D^*$.
Then $\psi(t^b\otimes(b+\sigma)\wedge v)=t^b\otimes(b+\sigma)\wedge v'$ for some $v'\in\wedge^{j-1}\bar D^*$. Since $i<j$, there exists $a\in G$ and $d\in\bar D$ such that $$(b+\sigma)(d)=0,(a\otimes d)\cdot (b+\sigma)\wedge v=0,(a\otimes d)\cdot (b+\sigma)\wedge v'\neq 0.$$
Then $t^ad\cdot t^b\otimes(b+\sigma)\wedge v=0$ and $t^ad\cdot t^b\otimes(b+\sigma)\wedge v'\neq 0$. This is a contradiction. Thus (1) is proved.

(2) It is easy to see that the linear map
$$\tau:\Gamma(V,\sigma)\rightarrow\Gamma(V,\sigma+a),t^b\otimes v\mapsto t^{b-a}\otimes v$$
is a $W$-module isomorphism for any $a\in G$. $V\cong V'$ and $\sigma-\sigma'\in G$, then $\Gamma(V,\sigma)\cong \Gamma(V',\sigma')$ as $W$-modules. Now suppose $r=1$, $\sigma-\sigma'\in G,\sigma\notin G$, $V\cong\wedge^0\bar D^*$ and $V'\cong\wedge^r\bar D^*$. Then $V=\C v, V'=\C v'$ and the $W$ action is given by
$$t^ad\cdot(t^b\otimes v)=(\sigma+b)(d)t^{a+b}\otimes v,\ t^ad\cdot(t^b\otimes v')=(\sigma+a+b)(d)t^{a+b}\otimes v'.$$
Let $d\in\bar D\setminus\{0\}$ and define $\psi:\Gamma(V,\sigma)\rightarrow\Gamma(V',\sigma)$ by $\psi(t^b\otimes v)=(\sigma+b)(d)t^b\otimes v'$. It is straight forward to verify that $\psi$ is an isomorphism of $W$-modules.
Then the sufficiency is proved.

If $\sigma-\sigma'\notin G$, then $\Gamma(V,\sigma)$ and $\Gamma(V',\sigma')$ have different support sets, which means that $\Gamma(V,\sigma)\ncong\Gamma(V',\sigma')$. To prove the necessity, we can assume that $\sigma=\sigma'$ without loss of generality.

Suppose $\psi:\Gamma(V,\sigma)\rightarrow\Gamma(V',\sigma)$ is a $W$-module isomorphism. Let $v\in V,v'\in V'$ with $\psi(1\otimes v)=1\otimes v'$. Then for any $a,b\in G,d_1,d_2\in D,k\in\Z$ we have
\begin{equation*}\begin{split}
&t^{-ka}d_1\cdot t^{ka+b}d_2\cdot 1\otimes v\\
=&t^{-ka}d_1\cdot t^{ka+b}\otimes(\sigma(d_2)+(ka+b)\otimes d_2)v\\
=&t^b\otimes((ka+b)(d_1)+\sigma(d_1)-ka\otimes d_1)(\sigma(d_2)+(ka+b)\otimes d_2)v\\
=&k^2t^b\otimes(a(d_1)-a\otimes d_1)(a\otimes d_2)v+kx+y,
\end{split}\end{equation*}
and
$$t^{-ka}d_1\cdot t^{ka+b}d_2\cdot 1\otimes v'=k^2t^b\otimes(a(d_1)-a\otimes d_1)(a\otimes d_2)v'+kx'+y',$$
where $x,y,x',y'$ are determined by $a,b,d_1,d_2$ and independent of $k$. So
$$\psi(t^b\otimes(a(d_1)-a\otimes d_1)(a\otimes d_2)v)=t^b\otimes(a(d_1)-a\otimes d_1)(a\otimes d_2)v'.$$

Suppose $(a(d_1)-a\otimes d_1)(a\otimes d_2)V=0$ for all $a\in G,d_1,d_2\in D$. By Remark \ref{rmkglr}, there exists $a_1,\dots,a_r\in G$ and $d_1,\dots,d_r\in \bar D$ with $a_i(d_j)=\delta_{i,j}$. Then $(1-a_i\otimes d_i)(a_i\otimes d_i)V=0$. So any weight of $V$ with respect to $a_i\otimes d_i$ is $0$ or $1$.
Then $V$ is isomorphic to $\wedge^l\bar D^*$ for some $l\in\{0,\dots,r\}$, and so is $V'$. Let $V\cong\wedge^i\bar D^*$ and $V'\cong\wedge^j\bar D^*$ with $i,j\in\{0,\dots,r\}$. Since $\dim V=\dim V'$, we have $i=j$ or $i+j=r$. If $i=j$ then we are done. Suppose $i\neq j$. Then $i+j=r$. If $\{i,j\}\neq\{0,r\}$, then, by Theorem \ref{simpleGamma}(3), the nontrivial simple subquotients of $\Gamma(\wedge^i\bar D^*,\sigma)$ are $\Gamma(\sigma,i),\Gamma(\sigma,i+1)$, and the nontrivial simple subquotients of $\Gamma(\wedge^j\bar D^*,\sigma)$ are $\Gamma(\sigma,j),\Gamma(\sigma,j+1)$. Then, by (1), $\Gamma(\wedge^i\bar D^*,\sigma)\ncong\Gamma(\wedge^j\bar D^*,\sigma)$ if $i\neq j$.
If $\{i,j\}=\{0,r\}$, assume without loss of generality that $i=0$ and $j=r$. Then $\dim V=\dim V'=1$ and for any $a,b\in G$ we have
$$t^ad\cdot(t^b\otimes v)=(\sigma+b)(d)t^{a+b}\otimes v,\ t^ad\cdot(t^b\otimes v')=(\sigma+a+b)(d)t^{a+b}\otimes v'.$$
If $r>1$ or $\sigma\in G$, there exists $a,b\in G$ such that $b\notin\C(\sigma+b)$. Let $d\in D$ with $(\sigma+b)(d)=0$ and $(\sigma+a+b)(d)\neq 0$. Then $t^ad\cdot(t^b\otimes v)=0$ and $t^ad\cdot(t^b\otimes v')\neq 0$. Note that $\psi(t^b\otimes v)$ is a nonzero multiple of $t^b\otimes v'$ in this case. So this is a contradiction. Hence, $r=1$ and $\sigma\notin G$.

Suppose $(a(d_1)-a\otimes d_1)(a\otimes d_2)V\neq 0$ for some $a\in G$ and some $d_1,d_2\in D$. Choose $v$ and $v'$ such that $(a(d_1)-a\otimes d_1)(a\otimes d_2)v=w\neq 0$ and $(a(d_1)-a\otimes d_1)(a\otimes d_2)v'=w'\neq 0$. Then $\psi(t^b\otimes w)=t^b\otimes w'$ for all $b\in G$. For any $b,c\in G,d\in D$ we have $t^bd\cdot t^c\otimes w=t^{b+c}\otimes(d(c)+\sigma(d)+b\otimes d)w$ and $t^bd\cdot t^c\otimes w'=t^{b+c}\otimes(d(c)+\sigma(d)+b\otimes d)w'$. Then $\psi(t^{b+c}\otimes(b\otimes d)w)=t^{b+c}\otimes(b\otimes d)w'$. So $\psi(t^b\otimes(x\cdot w))=t^b\otimes(x\cdot w')$ for all $b\in G,x\in U(L)$. Therefore, the map
$$V\rightarrow V',x\cdot w\mapsto x\cdot w',x\in U(L)$$
is a well-defined $L$-module isomorphism.

(3) Suppose $\psi:\Gamma(V,\sigma)\rightarrow\Gamma(\sigma,i)$ is a $W$-module isomorphism. Clearly, $\sigma-\sigma'\in G$. By the proof of (2), for any $v\in V,v'\in\wedge^i\bar D^*$ with $\psi(1\otimes v)=1\otimes v'$ and any $a,b\in G,d_1,d_2\in D$, we have
$$\psi(t^b\otimes(a(d_1)-a\otimes d_1)(a\otimes d_2)v)=t^b\otimes(a(d_1)-a\otimes d_1)(a\otimes d_2)v'=0.$$
So $(a(d_1)-a\otimes d_1)(a\otimes d_2)V=0$ for all $a\in G,d_1,d_2\in D$. Then $V\cong\wedge^j\bar D^*$ for some $j\in\{0,\dots,r\}$. By Theorem \ref{simpleGamma}, $\sigma\notin G$ and $j=0$ or $r$. Note that $\Gamma(\wedge^0\bar D^*,\sigma)\cong\Gamma(\sigma',1)$ and $\Gamma(\wedge^r\bar D^*,\sigma)\cong\Gamma(\sigma',r)$ if $\sigma-\sigma'\in G,\sigma\notin G$. Thus (3) follows from (1).
\end{proof}

\begin{lemma}
Any one dimensional $L$-module is of the form $V^c$ for some $c\in\C$, where $V^c=\C v$ with $a\otimes d\cdot v=c\phi(a,d)v,a\in G,d\in D$.
\end{lemma}
\begin{proof}
Note that for any $c\in\C$, $V^c$ is indeed an $L$-module.

Suppose $G$ is free of finite rank. By Lemma \ref{Rmodule}, the one dimensional $L$-modules are unique up to a scalar multiplication of the action of $L$. Then any one dimensional $L$-module is isomorphic to $V^c$ for some $c$.

Now suppose $G$ is not free of finite rank and $V=\C v$ is a one dimensional $L$-module. Let $G_0$ be a nonzero subgroup of $G$ that is free of finite rank, then $V$ is a one dimensional $G_0\otimes_\Z D$-module. So there is a $c\in\C$ such that $a\otimes d\cdot v=c\phi(a,d)v$ for all $a\in G_0,d\in D$. For any $a_0\in G$, $V$ is also a one dimensional $(G_0+\Z a_0)\otimes_\Z D$-module. Then $a_0\otimes d\cdot v=c\phi(a_0,d)v$. Thus, $V\cong V^c$ as $L$-modules.
\end{proof}

It is easy to see that $\Gamma(V^c,\sigma)$ is simple if $\sigma\notin G$ or $c\notin\{0,1\}$.

Suppose $\dim\bar D=\infty$.
For $\sigma\in G$, let
\begin{equation}\Gamma(\sigma,1)=\Gamma(V^0,\sigma)/\C t^{-\sigma}\otimes V^0,\ \Gamma(\sigma,\infty)=\oplus_{a\in G,a\neq-\sigma}\C t^a\otimes V^1\subset\Gamma(V^1,\sigma).\end{equation}
Then $\Gamma(\sigma,1)$ and $\Gamma(\sigma,\infty)$ are simple $W$-modules.

\begin{lemma}
Suppose $\dim\bar D=\infty$. Let $\sigma\in D^*$ and $c,c'\in\C$. Then
\begin{itemize}
\item[(1)] $\Gamma(V^c,\sigma)\cong\Gamma(V^{c'},\sigma)$ if and only if $c=c'$;
\item[(2)] if $\sigma\in G$, then any two modules in $\Gamma(V^c,\sigma),\Gamma(\sigma,1),\Gamma(\sigma,\infty)$ are not isomorphic.
\end{itemize}
\end{lemma}
\begin{proof}
(1) Suppose $c\neq c'$ and $\Gamma(V^c,\sigma)\cong\Gamma(V^{c'},\sigma)$. Let $V^c=\C v,V^{c'}=\C v'$. Since $\dim\bar D=\infty$ and $c\neq c'$, there exists $a,b\in G$ such that $b+\sigma+ca$ and $b+\sigma+c'a$ are $\C$-linearly independent. Choose $d\in D$ such that $(b+\sigma+ca)(d)\neq 0$ and $(b+\sigma+c'a)(d)=0$. Then $t^ad\cdot t^b\otimes v=t^{a+b}\otimes(b+\sigma+ca)(d)v\neq 0$ and $t^ad\cdot t^b\otimes v'=t^{a+b}\otimes(b+\sigma+c'a)(d)v'=0$. This is a contradiction.

(2) By comparing the support sets we have $\Gamma(V^c,\sigma)\ncong\Gamma(\sigma,1)$ and $\Gamma(V^c,\sigma)\ncong\Gamma(\sigma,\infty)$. Since $\dim\bar D=\infty$, there exists $a,b\in G$ with $b\neq \sigma$, and $b+\sigma$ and $b+\sigma+a$ are $\C$-linearly independent. Similar to the proof of (1), we can prove that $\Gamma(\sigma,1)\ncong\Gamma(\sigma,\infty)$.
\end{proof}

\section{Classification of simple bounded modules}\label{boundedmodule}
In this section we will classify simple bounded modules.
We will first prove that any simple bounded $AW$-module is of the form $\Gamma(V,\sigma)$ for a simple $L$-module $V$.

Let $\mathcal J$ be the left ideal of $U(\tilde W)$ generated by all elements of the from $t^a\cdot t^b-t^{a+b},t^0-1,a,b\in G$. Then $\mathcal J$ is in fact an ideal of $U(\tilde W)$. Let $\bar{U}=U(\tilde{W})/\mathcal J$. Identify $A$ and $W$ with their images in $\bar U$, then $\bar U=A\cdot U(W)$.

Let
$$T=\span\{t^{-a}(t^ad)-d\ |\ a\in G,d\in D\}\subset\bar U.$$
Then $T$ is a Lie subalgebra of $\bar U$, and its bracket is given by
\begin{equation}\label{bracket in T}
\begin{split}
&[t^{-a_1}(t^{a_1}d_1)-d_1, t^{-a_2}(t^{a_2}d_2)-d_2]\\
=&d_1(a_2)((t^{-a_1-a_2}(t^{a_1+a_2}d_2)-d_2)-(t^{-a_2}(t^{a_2}d_2)-d_2))-\\
&d_2(a_1)((t^{-a_1-a_2}(t^{a_1+a_2}d_1)-d_1)-(t^{-a_1}(t^{a_1}d_1)-d_1)),a_1,a_2\in G,d_1,d_2\in D.
\end{split}
\end{equation}
Let $\mathcal K$ be the associative subalgebra of $\bar U$ generated by $A$ and $D$.
\begin{lemma}\label{iota}
The map
$$\iota:\mathcal K\otimes U(T)\rightarrow\bar U,x\otimes y\mapsto x\cdot y,x\in\mathcal K,y\in U(T)$$
is an associative algebra isomorphism.
\end{lemma}
\begin{proof}
Since $T$ is a Lie subalgebra of $\bar U$, the restriction of $\iota$ on $U(T)$ is well-defined. Note that $\iota(\mathcal K)$ and $\iota(U(T))$ are commutative. Then $\iota$ is an associative algebra homomorphism.

Clearly, $A\subset\mathrm{Im}\ \iota$. For any $b\in G,d\in D$, we have
$$t^bd=t^b\cdot(t^{-b}(t^bd)-d)+t^b\cdot d\in\rm{Im}_\iota.$$
So $\iota$ is surjective.
For any eigenvector $v\in\bar U$ of $D$ of eigenvalue $a\in G$, the preimage of $v$ in $\mathcal K\otimes U(T)$ is of the form $t^a\cdot x\otimes y$, where $x\in U(D),y\in U(T)$. By the PBW theorem, if $\iota(t^a\cdot x\otimes y)=0$, then $x=0$ or $y=0$. So $\iota$ is injective and thus is an associative algebra isomorphism.
\end{proof}

Let $\m_G=\span\{t^a-1\ |\ a\in G\}$, which is an ideal of $A$. Then $\m_GD$ is a subalgebra of $W$.
For any $k\in\N$, $\m_G^kD$ is an ideal of $\m_GD$. Let $I_k=\m_G^kD$.

\begin{lemma}\label{T m D}
\begin{itemize}
\item[(1)]The map
$$\theta:T\rightarrow\m_GD,t^{-a}(t^ad)-d\mapsto(t^a-1)d$$
is a Lie algebra isomorphism.
\item[(2)]The map
$$\pi:L=G\otimes_{\Z}D\rightarrow\m_GD/I_2,a\otimes d\mapsto(t^a-1)d+I_2$$
is a Lie algebra isomorphism.
\end{itemize}
\end{lemma}
\begin{proof}
(1) Clearly, $\theta$ maps a basis of $T$ to a basis of $\m_GD$.
It is straight forward to verify that $\theta$ is an isomorphism of Lie algebras.

(2) For any $a_1, a_2\in G$, we have
$$(t^{a_1}-1)(t^{a_2}-1)=(t^{a_1+a_2}-1)-(t^{a_1}-1)-(t^{a_2}-1).$$
So
\begin{equation*}
(t^{a_1+a_2}-1)d+I_2=(t^{a_1}-1)d+(t^{a_2}-1)d+I_2,(t^a-1)d+I_2=-(t^{-a}-1)d+I_2
\end{equation*}
for all $d\in D$. Thus $\pi$ is a well-defined epimorphism. Also, the bracket in $\m_GD/I_2$ is given by
$$[(t^{a_1}-1)d_1+I_2,(t^{a_2}-1)d_2+I_2]=d_1(a_2)(t^{a_1}-1)d_2-d_2(a_1)(t^{a_2}-1)d_1+I_2.$$
Then it is easy to see that $\pi$ is an epimorphism of Lie algebras. Also, it is easy to verify that $\pi^{-1}(I_2)=0$. So $\pi$ is an isomorphism.
\end{proof}

\begin{lemma}\label{mi+j<mimj}
$I_{i+j}\subset[I_i,I_j]\subset I_{i+j-1}$ for all $i,j\in\N$.
\end{lemma}
\begin{proof}
Note that $x(\m_G^k)\subset\m_G^{k-1}$ for all $x\in W,k>1$. Then for any $f\in\m_G^i,g\in\m_G^j,d_1,d_2\in D$, we have
\begin{equation*}
[fd_1,gd_2]=fd_1(g)d_2-gd_2(f)d_1\in \m_G^{i+j-1}D.
\end{equation*}
So $[I_i,I_j]\subset I_{i+j-1}$.

For any $a\in G$, we have
\begin{equation*}\begin{split}
[t^{-a}fd_1,t^agd_2]=&d_1(a)fgd_2+d_2(a)fgd_1+fd_1(g)d_2-gd_2(f)d_1\\
=&d_1(a)fgd_2+d_2(a)fgd_1+[fd_1,gd_2].
\end{split}\end{equation*}
So $d_1(a)fgd_2+d_2(a)fgd_1\in[I_i,I_j]$.
Let $d'\in\bar D,d''\in\Ker_2\phi$. Then there is an $a'\in G$ with $d'(a')\neq 0$. Taking $a=a'$ and $d_1=d_2=d'$, we have $2d'(a')fgd'\in[I_i,I_j]$. So $fg\bar D\subset[I_i,I_j]$. Taking $a=a',d_1=d'$ and $d_2=d''$, we have $d'(a')fgd''\in[I_i,I_j]$. So $fg\Ker_2\phi\subset[I_i,I_j]$. Therefore, $I_{i+j}\subset[I_i,I_j]$.
\end{proof}

\begin{lemma}\label{GsD1}
Suppose $\dim D=1,G\cong\Z^s$. Then
\begin{itemize}
\item[(1)]any finite dimensional simple $\m_G\otimes D$-module $V$ is of dimension $1$ and $I_2\cdot V=0$;
\item[(2)]if $V$ is a finite dimensional $\m_G\otimes D$-module with $\dim V=N$, then $I_2^N\cdot V=I_{2N}\cdot V=0$.
\end{itemize}
\end{lemma}
\begin{proof}
(1)Since $\dim D=1$, $\phi$ is non-degenerate. Then $W$ is isomorphic to a solenoidal Lie algebra defined in \cite{BF2}. For any finite dimensional simple $\m_G\otimes D$-module $V$, the $\mathcal K\otimes U(\m_G\otimes D)$-module $A\otimes V$ becomes an $AW$-module under the associative algebra isomorphism $\iota^{-1}\circ(1\otimes\theta):\bar U\rightarrow\mathcal K\otimes U(\m_G\otimes D)$. By Corollary 4.7 in \cite{BF2}, any simple bounded $AW$-module are isomorphic to $\Gamma(\C,\sigma)$ for some one dimensional $L$-module $\C$ and some $\sigma\in D^*$. Then $\dim V=1$. Suppose $V=\C v$, then there exists $c\in\C$ such that $(t^a-1)d\cdot v=cd(a)v$ for any $a\in G,d\in D$.
For any $a_1,a_2\in G,d\in D$, we have
\begin{eqnarray*}(t^{a_1}-1)(t^{a_2}-1)d\cdot v&=&((t^{a_1+a_2}-1)d-(t^{a_1}-1)d-(t^{a_2}-1)d)\cdot v\\
&=&c(d(a_1+a_2)-d(a_1)-d(a_2))\cdot v=0.\end{eqnarray*}
So $I_2\cdot v=0$.

(2)Let $0=V_0\subset V_1\subset\dots\subset V_k=V$ be a series of $m_G\otimes D$-modules such that each $V_i/V_{i-1}$ is a simple $m_G\otimes D$-modules. By (1), $\dim V_i=\dim V_{i-1}+1$ and $I_2\cdot V_i\subset V_{i-1}$. So $k=N$ and $I_2^N\cdot V=0$. Then, by lemma \ref{mi+j<mimj}, $I_{2N}\cdot V\subset I_2^N\cdot V=0$.
\end{proof}

\begin{lemma}\label{d1dr}
Suppose that $\dim\bar{D}=r<\infty$ and $G\cong\Z^s$, then there exists an index set $I$ and a basis $\{d_i\ |\ i\in I\}$ of $D$ such that $\phi|_{G\times\C d_i}$ is non-degenerate for all $i\in I$.
\end{lemma}
\begin{proof}
Let $d_1',\dots,d_r'$ be a basis of $\bar{D}$ and $e_1,\dots,e_s$ be the standard basis of $G\cong\Z^s$. Let $B=(d_i'(e_j))_{r\times s}$. Since $\phi|_{G\times\bar{D}}$ is non-degenerate, for any column vector $v\in\Z^s$, $Bv=0$ if and only if $v=0$.

Let $\Q'$ be the subfield $\Q(d_i'(e_j)\ |\ 1\leqslant i\leqslant r, 1\leqslant j\leqslant s)$ of $\C$ and let $C=(c_{ij})$ be a $r\times r$ matrix such that the set $\{1,c_{ij}\ |\ 1\leqslant i,j\leqslant r\}$ is $\Q'$-linearly independent. Note that there exists an integer $t$ such that $|tE+C|\neq 0$, where $E$ is the $r\times r$ identity matrix. Without loss of generality, assume that $|C|\neq 0$.
Let $u_1,\dots,u_r$ be the row vectors of $C$. For any $v\in\Z^s$, we have $u_iBv=0$ if and only if $Bv=0$ if and only if $v=0$.

Let $\begin{bmatrix}d_1\\\vdots\\d_r\end{bmatrix}=C\begin{bmatrix}d_1'\\\vdots\\d_r'\end{bmatrix}$, then $d_1, ..., d_r$ is a basis of $\bar{D}$ with $\phi|_{G\times\C d_i}$ non-degenerate. Let $\{d_i'\ |\ i\in J\}$ be a basis of $\Ker_2\phi$, where $J$ is an index set.
Set $d_i=d_1+d_i'$ for all $i\in J$. Then $\{d_i\ |\ i\in I\}$ is a basis of $D$ with $\phi|_{G\times\C d_i}$ non-degenerate for all $i\in I$, where $I=\{1,\dots,r\}\cup J$.
\end{proof}

\begin{lemma}\label{GDr}
Suppose that $\dim\bar{D}=r<\infty$ and let $V$ be a $N$ dimensional module over $\m_G\otimes D$. Then
\begin{itemize}
\item[(1)]$I_2^N\cdot V=0$.
\item[(2)]if $V$ is a simple $\m_G\otimes D$-module, then $I_2\cdot V=0$.
\end{itemize}
\end{lemma}
\begin{proof}
(1)Suppose that (1) is true for $G\cong\Z^s$. For any $x\in I_2^N$, there exists a subgroup $G_1$ of $G$ such that $G_1\cong\Z^s$ for some $s$, and $x\in(\m_{G_1}^2\otimes D)^N$. Since $V$ is a $N$-dimensional $\m_{G_1}\otimes D$-module, $x\cdot V=0$. Thus $I_2^N\cdot V=0$.

Now we prove the case that $G\cong\Z^s$.

Let $I$ and $d_i,i\in I$ be as in Lemma \ref{d1dr}. By Lemma \ref{GsD1}, $(\m_G^{2N}\otimes\C d_i)\cdot V=0,i=1,\dots,r$. Hence $(\m_G^{2N}\otimes D)\cdot V=0$, and $V$ could be viewed as a $\m_G\otimes D/\m_G^{2N}\otimes D$-module.

By Lemma \ref{mi+j<mimj}, $I_2/\m_G^{2N}\otimes D$ is a nilpotent Lie algebra. By Lie's theorem, any finite dimensional simple module over $I_2/\m_G^{2N}\otimes D$ is of dimension $1$. Let
$$0=V_0\subset V_1\subset\dots\subset V_N=V$$
be a series of $I_2/\m_G^{2N}\otimes D$-modules with all $V_i/V_{i-1}$ simple. By Lemma \ref{GsD1}, for any $x\in I_2,x^N\cdot V=0$. Then $0$ is the only eigenvalue of the action of $x$ on $V$, and therefore $x\cdot V_i\subset V_{i-1}$ for all $i$. So $I_2^N\cdot V=0$.

(2)If $V$ is simple, then $I_2\cdot V=0$ or $V$. Since $I_2^N\cdot V=0$, $I_2\cdot V=0$.
\end{proof}

\begin{theorem}\label{simpleAW}
Suppose $\dim\bar{D}=r<\infty$. Then any simple bounded $AW$-module $M$ is isomorphic to $\Gamma(V,\alpha)$ for some finite dimensional simple $L$-module $V$ and some $\alpha\in\supp(M)$.
\end{theorem}
\begin{proof}
Let $V=M_\alpha$. Then by Lemma \ref{iota}, for any nonzero $v\in V$,
$$M=\bar Uv=\mathcal KU(T)v=AU(T)U(D)v=AU(T)v.$$
So $V=U(T)v$ and therefore $V$ is a simple $T$-module. From Lemma \ref{T m D} and Lemma \ref{GDr}, we know that $V$ is a simple module over $L$ with the action given by
$$a\otimes d\cdot v=(t^{-a}\cdot t^ad-d)\cdot v,\forall a\in G,d\in D.$$

Clearly, the natural map $\Gamma(V,\alpha)=A\otimes V\rightarrow AV=M$ is an isomorphism of spaces. 
It is straight forward to verify that this map is an $AW$-module isomorphism.
\end{proof}

For any $W$-module $\Gamma$, let $\hat\Gamma=\span\{\psi(t^ad,v)\ |\ a\in G,d\in D,v\in\Gamma\}\subset\mathrm{Hom}_\C(A,\Gamma)$, where
$$\psi(t^ad,v)(t^b)=t^{a+b}d\cdot v,\forall b\in G.$$
$\hat\Gamma$ has a $\tilde W$-module structure given by
$$t^{a'}d'\cdot\psi(t^ad,v)=\psi(t^ad,t^{a'}d'\cdot v)+\psi([t^{a'}d',t^ad],v),$$
$$t^{a'}\cdot\psi(t^ad,v)=\psi(t^{a+a'}d,v),\ a,a'\in G,d,d'\in D,v\in\Gamma.$$
It is easy to see that $\hat\Gamma$ is an $AW$-module. $\hat\Gamma$ is called the $A$-cover of $\Gamma$.

\begin{lemma}\label{cover}
Let $\Gamma$ be a $W$-module. Then the map $\eta:\hat{\Gamma}\rightarrow\Gamma,f\mapsto f(1)$ is a homomorphism of $W$-modules and $\eta(\hat{\Gamma})=W\cdot\Gamma$.
\end{lemma}
\begin{proof}
For any $a,a'\in G,d,d'\in D$ and any $v\in\Gamma$, we have
$$t^{a'}d'\cdot\psi(t^ad,v)(1)=[t^{a'}d',t^ad]\cdot v+t^ad\cdot t^{a'}d'\cdot v=(t^{a'}d'\cdot\psi(t^ad,v))(1).$$
Then $\eta$ is a $W$-module homomorphism. Clearly, $\eta(\hat\Gamma)=W\cdot\Gamma$.
\end{proof}

\begin{theorem}\label{cuspidal}
Suppose $G\cong\Z^s$ and $\dim\bar{D}=r<\infty$. Let $\Gamma$ be a nontrivial simple bounded $W$-module. Then $\Gamma$ is a simple $W$-quotient of $\Gamma(V,\sigma)$ for some finite dimensional simple $L$-module $V$ and some $\sigma\in\supp(\Gamma)$.
\end{theorem}
\begin{proof}
Let $I$ and $d_i,i\in I,$ be as in Lemma \ref{d1dr}.
For any $d\in D$, let
$$\hat\Gamma(d)=\span\{\psi(t^ad,v)\ |\ a\in G,v\in\Gamma\}.$$
Then $\hat\Gamma(d_i)$ is the $A$-cover of the $W(\phi|_{G\times\C d_i})$-module $\Gamma$ for all $i\in I$. By Theorem 4.7 in \cite{BF1}, $\hat\Gamma(d_i)$ is a bounded $AW(\phi|_{G\times\C d_i})$-module.

Claim. There is a bounded $AW$-submodule $\hat{\Gamma}'$ of $\hat{\Gamma}$ such that $\eta(\hat{\Gamma}')=\Gamma$.

Case 1.  $W(\phi|_{G\times\Ker_2\phi})\cdot\Gamma=0$.

Then for any $d\in\Ker_2\phi,v\in\Gamma,a\in G$, we have $\psi(t^ad,v)=0$. Then $\hat\Gamma=\sum_{i=1}^r\hat\Gamma(d_i)$, where $d_1,\dots,d_r$ is a basis of $\bar D$ with $\phi|_{G\times\C d_i}$ non-degenerate for all $i$. Since each $\hat\Gamma(d_i)$ is a bounded module, $\hat\Gamma$ is bounded. Clearly, $\eta(\hat\Gamma)=W\cdot\Gamma=\Gamma$.

Case 2.  $W(\phi|_{G\times\Ker_2\phi})\cdot\Gamma\neq 0$.

Let $d\in\Ker_2\phi$ with $W(\phi|_{G\times\C d})\cdot\Gamma\neq 0$. Then $W(\phi|_{G\times\C d})$ is an ideal of $W$ and $\hat\Gamma(d)$ is an $AW$-submodule of $\hat\Gamma$. By the simplicity of $\Gamma$, $\eta(\hat\Gamma(d))=W(\phi|_{G\times\C d})\cdot\Gamma=\Gamma$. Note that there is a finite subset $J$ of $I$ such that $d\in\span\{d_i\ |\ i\in J\}$. Then $\hat\Gamma(d)\subset\sum_{i\in J}\hat\Gamma(d_i)$. Since each $\hat\Gamma(d_i)$ is a bounded module, $\hat\Gamma(d)$ is bounded.

Hence the Claim is proved.

Now let $0=\hat{\Gamma}'_0\subset\hat{\Gamma}'_1\subset\dots\subset\hat{\Gamma}'_l=\hat{\Gamma}'$ be a series of $AW$-modules with $\hat{\Gamma}'_i/\hat{\Gamma}'_{i-1}$ simple for all $i\in\{1,\dots,l\}$. Let $k$ be the smallest integer such that $\eta(\hat\Gamma'_{k-1})=0,\eta(\hat\Gamma'_k)=\Gamma$. Then $\Gamma$ is a simple $W$-quotient of simple $AW$-module $\hat{\Gamma}'_k/\hat{\Gamma}'_{k-1}$, which, by Theorem \ref{simpleAW}, is of the form $\Gamma(V,\sigma)$ as stated.
\end{proof}

This theorem, together with Corollary \ref{simplequotient}, gives a classification of simple bounded $W$-module when $G\cong\Z^s$ and $\dim \bar{D}= r<\infty$.

\begin{lemma}\label{samedim}
Let $\Gamma$ be a simple bounded $W$-module and $\lambda\in\supp(\Gamma)$. Then
\begin{itemize}
\item[(1)]for any $\mu\in\supp(\Gamma),\dim\Gamma_{\lambda}=\dim\Gamma_{\mu}$.
\item[(2)]there is a subgroup $G_0$ of $G$ with $G_0\cong\Z^n$ such that $\Gamma_{\lambda+G_0}=\oplus_{a\in G_0}\Gamma_{\lambda+a}$ is a simple $W(\phi|_{G_0\times D})$-module.
Moreover, for any subgroup $G_1$ of $G$ with $G_1\cong\Z^m$ and $G_0\leqslant G_1$, $\Gamma_{\lambda+G_1}$ is a simple $W(\phi|_{G_1\times D})$-module.
\end{itemize}
\end{lemma}
\begin{proof}
(1)By Lemma \ref{G1D1}, there is a subgroup $G_0$ of $G$ such that $G_0\cong\Z^n$, $\lambda-\mu\in G_0$, and $\Gamma$ has a simple $W(\phi|_{G_0\times D})$-subquotient $\Gamma'$ with $\dim\Gamma_{\lambda}=\dim\Gamma'_{\lambda},\dim\Gamma_{\mu}=\dim\Gamma'_{\mu}$. Since $G_0$ is of the finite rank, the codimension of $\Ker_2\phi|_{G_0\times D}$ in $D$ is finite. By Theorem \ref{cuspidal} and Corollary \ref{cussamedim}, the weight spaces of $\Gamma'$ are of the same dimension. So $\dim\Gamma_{\lambda}=\dim\Gamma'_{\lambda}=\dim\Gamma'_{\mu}=\dim\Gamma_{\mu}$.

(2)Let $\mu\in\supp(\Gamma)$ with $\mu=0$ if $0\in\supp(\Gamma)$ and let $G_0$ and $\Gamma'$ be as in the proof of (1). Then for any $\lambda'\in\supp(\Gamma')$ we have $\dim\Gamma'_{\lambda'}=\dim\Gamma'_\lambda=\dim\Gamma_\lambda=\dim\Gamma_{\lambda'}$. So $\Gamma'_{\lambda'}=\Gamma_{\lambda'}$. Thus $\Gamma_{\lambda+G_0}=\Gamma'$ is a simple $W(\phi|_{G_0\times D})$-module.

Let $G_1\leqslant G$ with $G_1\cong\Z^m$ and $G_0\leqslant G_1$.
By Corollary \ref{cussamedim}, for any simple $W(\phi|_{G_1\times D})$-submodule $M$ of $\Gamma_{\lambda+G_1}$, $\supp(M)=\{0\},\lambda+G_1$ or $G_1\setminus\{0\}$. Then $M\cap\Gamma_{\lambda+G_0}\neq 0$. Since $G_0\leqslant G_1$ and $\Gamma_{\lambda+G_0}$ is a simple $W(\phi|_{G_0\times D})$-module, $\Gamma_{\lambda+G_0}\subset M$. Then $M=\Gamma_{\lambda+G_1}$ since the weight spaces of $M$ are of the same dimension.
\end{proof}

\begin{theorem}\label{bounded1}
Suppose $\dim\bar{D}= r<\infty$. Let $\Gamma$ be a simple bounded $W$-module. Then $\Gamma$ is isomorphic to a simple quotient of $\Gamma(V,\sigma)$, where $V$ is a finite dimensional simple $L$-module and $\sigma\in D^*$.
\end{theorem}
\begin{proof}
Let $\sigma\in\supp(\Gamma)$ and $G_0\leqslant G$ be as in Lemma \ref{samedim}(2). Replace $G_0$ with $G_0+\Z\sigma$ if $\sigma\in G$. By Lemma \ref{fin*fin}, without loss of generality, assume that $\phi|_{G_0\times\bar D}$ is non-degenerate.
Then by Theorem \ref{cuspidal} and Corollary \ref{simplequotient}, $\Gamma_{\sigma+G_0}$ is isomorphic to a simple quotient of $\Gamma_{W(\phi|_{G_0\times D})}(V,\sigma)$, where $V$ is a finite dimensional simple $G_0\otimes_{\Z}D$-module, $\sigma\in D^*$.

For any subgroup $G_1$ of $G$ with $G_1\cong\Z^m$ and $G_0\leqslant G_1$, we have $\Gamma_{\sigma+G_1}$ is a simple $W(\phi|_{G_1\times D})$-module by Lemma \ref{samedim}(2). Similarly, by Theorem \ref{cuspidal} and Corollary \ref{simplequotient}, $\Gamma_{\sigma+G_1}$ is isomorphic to a simple quotient of $\Gamma_{W(\phi|_{G_1\times D})}(V',\sigma')$, where $V'$ is a finite dimensional simple $G_1\otimes_{\Z}D$-module, $\sigma'\in D^*$.

By Lemma \ref{ncong}, it is easy to see that if $\Gamma_{\sigma+G_0}$ is isomorphic to $\Gamma_{W(\phi|_{G_0\times D})}(V,\sigma)$ (resp. $\Gamma_{W(\phi|_{G_0\times D})}(\sigma,k)$, the trivial $W(\phi|_{G_0\times D})$-module), then $\Gamma_{\sigma+G_1}$ is isomorphic to $\Gamma_{W(\phi|_{G_1\times D})}(V,\sigma)$ (resp. $\Gamma_{W(\phi|_{G_1\times D})}(\sigma,k)$, the trivial $W(\phi|_{G_1\times D})$-module).

Now suppose that $\Gamma_{\sigma+G_0}\cong\Gamma_{W(\phi|_{G_0\times D})}(V,\sigma)$, we will prove that $\Gamma\cong\Gamma(V,\sigma)$ as $W$-modules. (other cases could be proved similarly). Then $\Gamma_{\sigma+G_1}\cong\Gamma_{W(\phi|_{G_1\times D})}(V,\sigma)$ for all subgroups $G_1$ of $G$ with $G_1\cong\Z^m$ and $G_0\leqslant G_1$. For any such $G_1$, there is a natural $W(\phi_{G_1\times D})$-module monomorphism $\tau_{G_1}:\Gamma_{\sigma+G_1}\rightarrow\Gamma(V,\sigma)$. Since $\Gamma_{\sigma+G_0}$ is a simple $W(\phi_{G_0\times D})$-module, the map $\tau_{G_1}^{-1}\circ\tau_{G_0}:\Gamma_{\sigma+G_0}\rightarrow\Gamma_{\sigma+G_0}$ is a nonzero scalar multiplication. So we may assume that $\tau_{G_1}^{-1}\circ\tau_{G_0}=\id$, i. e., $\tau_{G_1}|_{\Gamma_{\sigma+G_0}}=\tau_{G_0}$ for all $G_1$. Define $\tau:\Gamma\rightarrow\Gamma(V,\sigma)$ by $\tau|_{\Gamma_{\sigma+G_1}}=\tau_{G_1}$ for all subgroups $G_1$ of $G$ with $G_1\cong\Z^m$ and $G_0\leqslant G_1$. Then $\tau$ is well-defined. In fact, if $\mu\in(\sigma+G_1)\cap(\sigma+G_2)$, then the map $\tau_{G_2}^{-1}\circ\tau_{G_1}:\Gamma_{\sigma+G_1\cap G_2}\rightarrow\Gamma_{\sigma+G_1\cap G_2}$ is a nonzero scalar multiplication. Since $\tau_{G_2}^{-1}\circ\tau_{G_1}|_{\Gamma_{\sigma+G_0}}=\id$, we have $\tau_{G_2}^{-1}\circ\tau_{G_1}|_{\Gamma_{\sigma+G_1\cap G_2}}=\id$. So $\Gamma$ is well-defined. It is easy to see that $\tau$ is an isomorphism of $W$-modules. So $\Gamma\cong\Gamma(V,\sigma)$.
\end{proof}

\begin{theorem}\label{cus1dim}
Suppose $\dim\bar{D}=\infty$. Let $\Gamma$ be a simple bounded $W$-module. Then $\dim\Gamma_{\lambda}=1$ for all $\lambda\in\supp(\Gamma)$.
\end{theorem}
\begin{proof}
Let $\lambda_0\in\supp(\Gamma)$ and $\dim\Gamma_{\lambda_0}=n$. By Lemma \ref{G1D1}, there is a subgroup $G_1$ of $G$ and a subspace $D_1$ of $D$ such that $G_1\cong\Z^s,n<\dim D_1<\infty$, and any $W(\phi|_{G_1\times D_1})$-submodule $M$ of $\Gamma$ with $M_{\lambda_0}\neq 0$ contains $\Gamma_{\lambda_0}$. Following from Lemma \ref{fin*fin}, by adding finite elements in $\bar D$ to $D_1$, we could assume that $\Ker_1\phi|_{G_1\times D_1}=0$. Clearly, $\Gamma'=X_1/X_2$ is a simple bounded $W(\phi|_{G_1\times D_1})$-module, where $X_1=U(W(\phi|_{G_1\times D_1}))\Gamma_{\lambda_0}$ and $X_2$ is the sum of all proper $W(\phi|_{G_1\times D_1})$-submodules of $X_1$. By Theorem \ref{bounded1} and Corollary \ref{cussamedim}, $\dim\Gamma'_{\lambda_0}=1$ or $\dim\Gamma'_{\lambda_0}\geqslant\dim D_1$. Since $\dim\Gamma'_{\lambda_0}=\dim\Gamma_{\lambda_0}=n<\dim D_1$, we have $n=1$.
\end{proof}

For any nonzero subgroup $G_0$ of $G$ that is free of finite rank, $(\Gamma(V^c,\sigma))_{\sigma+G_0}\cong\Gamma_{W(\phi|_{G_0\times D})}(V^c,\sigma)$ as $W(\phi|_{G_0\times D})$-modules. If $\sigma\in G$, then $(\Gamma(\sigma,1))_{\sigma+G_0}\cong\Gamma_{W(\phi|_{G_0\times D})}(\sigma,1)$ as $W$-modules. Let $\bar D_{G_0}$ be a subspace of $D$ such that $D=\bar D_{G_0}\oplus\Ker_2\phi|_{G_0\times D}$. Then if $\sigma\in G$, $(\Gamma(\sigma,\infty))_{\sigma+G_0}\cong\Gamma_{W(\phi|_{G_0\times D})}(\sigma,\dim\bar D_{G_0})$.
\begin{theorem}\label{bounded2}
Suppose $\dim\bar D=\infty$. Let $\Gamma$ be a simple bounded $W$-module. Then one of the following cases happens:
\begin{itemize}
\item[(1)]$\Gamma\cong\Gamma(V^c,\sigma)$, where $c\in\C,\sigma\in D^*$ with $\sigma\notin G$ or $c\notin\{0,1\}$.
\item[(2)]$\Gamma\cong\Gamma(\sigma,1)$ or $\Gamma(\sigma,\infty)$, where $\sigma\in G$.
\item[(3)]$\Gamma$ is the one dimensional trivial module.
\end{itemize}
\end{theorem}
\begin{proof}
The proof is similar to that of Theorem \ref{bounded1}.
\end{proof}

\section{Classification of simple Harish-Chandra modules}\label{SecHWT}
From now on, suppose that $\phi$ is non-degenerate.

Suppose $G=G'\oplus\Z a_0$ for some subgroup $G'$ of $G$ and $a_0\in G\setminus\{0\}$.
Let $W^0=\span\{t^ad\ |\ a\in G',d\in D\}, W^{\pm}=\span\{t^ad\ |\ a\in G'\pm\N a_0,d\in D\}$. Then $W^0,W^+,W^-$ are subalgebras of $W$ and $W=W^-\oplus W^0\oplus W^+$.
For any simple weight $W^0$-module $X$, we could view $X$ as a $W^0\oplus W^+$-module by setting $W^+X=0$.
Define the generalized Verma module $M(X,G',a_0)= U(W)\otimes_{U(W^0\oplus W^+)}X$. $M(X,G',a_0)$ is an indecomposable weight $W$-module and has a unique simple quotient $L(X,G',a_0)$.
$L(X, G', a_0)$ is called a module of highest weight type if $X$ is a bounded module.

By \cite{BZ}, in the case that $G\cong\Z^s,\dim D=r<\infty$, $L(X,G',a_0)$ is a Harish-Chandra module if and only if $X$ is a bounded module. Moreover, $L(X,G',a_0)$ is not bounded unless $X$ is a trivial module.

V. Mazorchuk and K. Zhao classified all simple Harish-Chandra $W_r$-modules that are not bounded in \cite{MZ}. Their proofs literally apply to such modules over $W$ if $G\cong\Z^s$, yielding the following result:
\begin{theorem}\label{highranknotcus}
Suppose $G\cong\Z^s,\dim D=r<\infty$. Let $\Gamma$ be a simple Harish-Chandra $W$-module that is not bounded. Then $\Gamma\cong L(X,G',a_0)$ for a subgroup $G'$ of $G$, an $a_0\in G\setminus\{0\}$ with $G=G'\oplus\Z a_0$, and a simple bounded $W^0$-module $X$.
\end{theorem}
Note that $W^0=W(G',D,\phi|_{G'\times D})$, over which the simple bounded modules are classified in Theorem \ref{bounded1}.

\begin{lemma}\label{cussubgroup}
Suppose $G\cong\Z^s$ and $\dim D=r<\infty$. Let $\Gamma=L(X,G_0,a_0)$ for some $G_0\leqslant G,a_0\in G\setminus\{0\}$ with $G=G_0\oplus\Z a_0$, and some simple bounded $W^0$-module $X$. Then $\Gamma_{\lambda+G_0}$ is bounded for all $\lambda\in\supp(\Gamma)$.
\end{lemma}
\begin{proof}
Let $\mu\in\supp(X)$ and choose $d_0\in D$ such that $\phi|_{G\times\C d_0}$ is non-degenerate. Let $\Gamma'$ be a simple $\Vir[G,d_0]$-subquotient of $\Gamma$, then $\supp(\Gamma')\subset\supp(\Gamma)\subset\mu-\Z_+a_0+G_0$. By Theorem \ref{bounded1} and \ref{highranknotcus}, $\Gamma'\cong L_{\Vir[G,d_0]}(Y,G_1, a_1)$ for some $G_1\leqslant G,a_1\in G\setminus\{0\}$ with $G=G_1\oplus\Z a_1$, and some simple bounded $\Vir[G_1,d_0]$-module $Y$. Let $\sigma=\mu+ka_0+g\in\supp(Y)$ for some $k\in\Z$ and $g\in G_0$. Then $$(\sigma+G_1)\setminus\{0\}=(\mu+ka_0+g+G_1)\setminus\{0\}\subset\supp(\Gamma')\subset\supp(\Gamma)\subset\mu-\Z_+a_0+G_0.$$
So $G_1\leqslant G_0$. Let $a_0=k_0a_1+g_1$ and $a_1=k_1a_0+g_0$ with $k_0,k_1\in\Z,g_0\in G_0,g_1\in G_1$. Then $a_0=k_0(k_1a_0+g_0)+g_1$. Since $G=G_0\oplus\Z a_0$, we have $k_0k_1=1$ and $k_0g_0+g_1=0$. Thus, $k_0=k_1$ and $G_0=G_1$. By Theorem 3.1 in \cite{BZ}, for any $\lambda\in\supp(\Gamma')$, $\Gamma'_{\lambda+G_0}$ is bounded.

If $\Gamma_{\lambda+G_0}$ has a simple $\Vir[G_0,d_0]$-subquotient that is not bounded, then $\Gamma$ has a simple $\Vir[G,d_0]$-subquotient $\Gamma'$ such that $\Gamma'_{\lambda+G_0}$ is not bounded. This is a contradiction. So any simple $\Vir[G_0,d_0]$-subquotient of $\Gamma_{\lambda+G_0}$ is bounded. By Theorem \ref{bounded1}, $\Gamma_{\lambda+G_0}$ has only finite simple $\Vir[G_0,d_0]$-subquotients. Then $\Gamma_{\lambda+G_0}$ is bounded.
\end{proof}

\begin{lemma}\label{VirZa}
Let $\Gamma$ be a simple Harish-Chandra $W$-module. For any $\lambda\in\supp(\Gamma),a\in G\setminus\{0\},d_0\in D$ with $a(d_0)\neq 0$, if the $\Vir[\Z a,d_0]$-module $\Gamma_{\lambda+\Z a}$ has a nontrivial bounded subquotient, then $\Gamma_{\lambda+\Z a}$ is bounded.
\end{lemma}
\begin{proof}
Assume that $\Gamma_{\lambda+\Z a}$ is not bounded. Then there are $\mu_1,\mu_2\in\supp(\Gamma_{\lambda+\Z a})$ such that $\dim\Gamma_{\mu_1}\neq\dim\Gamma_{\mu_2}$ and $\mu_i(d_0)\neq 0,i=1,2$. Let $I=\{\mu_1, \mu_2\}, I_G=\{a\}, I_D=\{d_0\}$, and let $G_1\times D_1\subset G\times D$ be as in Lemma \ref{G1D1}. From Lemma \ref{fin*fin}, assume that $\phi|_{G_1\times D_1}$ is non-degenerate by adding finite elements to $G_1$ and $D_1$ if necessary. Then the $W(\phi|_{G_1\times D_1})$-module $\Gamma_{\lambda+G_1}$ has a simple subquotient $\Gamma'$ with $\dim \Gamma'_{\mu_i}=\dim\Gamma_{\mu_i},i=1,2$. Since $\dim\Gamma'_{\mu_1}\neq\dim\Gamma'_{\mu_2}$, $\Gamma'$ is not bounded by Theorem \ref{bounded1}. By Theorem \ref{highranknotcus}, $\Gamma'\cong L_{W(\phi|_{G_1\times D_1})}(X,G_0,a_0)$ for some $G_0,a_0$ and $X$.

Note that the $\Vir[\Z a,d_0]$-module $\Gamma'_{\lambda+\Z a}$ is not bounded, otherwise $\dim\Gamma'_{\mu_1}\neq\dim\Gamma'_{\mu_2}$.
By Lemma \ref{cussubgroup}, $a\notin G_0$. Then $(\lambda+\Z a)\setminus\{0\}\not\subset\supp(\Gamma'_{\lambda+\Z a})$, which means that $\Gamma'_{\lambda+\Z a}$ does not have nontrivial bounded $\Vir[\Z a,d_0]$-subquotients. Since $\dim \Gamma'_{\mu_i}=\dim\Gamma_{\mu_i},i=1,2$, any simple nontrivial bounded $\Vir[\Z a,d_0]$-subquotient of $\Gamma_{\lambda+\Z a}$ is also a subquotient of $\Gamma'_{\lambda+\Z a}$. This is a contradiction.
\end{proof}

\begin{lemma}\label{G1+G2}
Let $\Gamma$ be a simple Harish-Chandra $W$-module, $\lambda\in\supp(\Gamma)\setminus\{0\}$, and $G_1,G_2\leqslant G$. If both $\Gamma_{\lambda+G_1},\Gamma_{\lambda+G_2}$ are bounded, then $\Gamma_{\lambda+G_1+G_2}$ is bounded.
\end{lemma}
\begin{proof}
For any $a\in(G_1\cup G_2)\setminus\{0\}$ with $\lambda+a\neq 0$, $\Gamma_{\lambda+\Z a}$ is bounded. Let $d_a\in D$ with $a(d_a),\lambda(d_a),(\lambda+a)(d_a)\neq 0$. Then $\Gamma_{\lambda+\Z a}$ is a bounded $\Vir[\Z a,d_a]$-module. It follows that $\dim\Gamma_\lambda=\dim\Gamma_{\lambda+a}$.

Suppose that $\Gamma_{\lambda+G_1+G_2}$ is not bounded. Then there is $a_1\in G_1,a_2\in G_2$ such that $\lambda+a_1+a_2\neq 0$ and $\dim\Gamma_{\lambda+a_1+a_2}\neq\dim\Gamma_{\lambda}$. Clearly, $a_1\notin G_2$ and $a_2\notin G_1$.

Let $I=\{\lambda+a_1,\lambda+a_2,\lambda+a_1+a_2\}$ and let $G_0\times D_0\subset G\times D$ satisfies (1)-(5) in Lemma \ref{G1D1}. Moreover, assume that $\phi|_{G_0\times D_0}$ is non-degenerate. Then $\Gamma_{\lambda+G_0}$ has a simple $W(\phi|_{G_0\times D_0})$-subquotient $\Gamma'$ such that $\dim \Gamma_{\lambda+a_1}=\dim \Gamma'_{\lambda+a_1},\dim \Gamma_{\lambda+a_2}=\dim \Gamma'_{\lambda+a_2},\dim \Gamma_{\lambda+a_1+a_2}=\dim \Gamma'_{\lambda+a_1+a_2}$.

Note that $\lambda+a_1\neq 0$ or $\lambda+a_2\neq 0$. Without loss of generality, assume that $\lambda+a_1\neq 0$. Then $\dim \Gamma'_{\lambda+a_1}=\dim \Gamma_{\lambda+a_1}=\dim \Gamma_{\lambda}\neq\dim\Gamma'_{\lambda+a_1+a_2}$. Then $\Gamma'$ is not a bounded $W(\phi|_{G_0\times D_0})$-module. By Theorem \ref{highranknotcus}, $\Gamma'\cong L_{W(\phi|_{G_0\times D_0})}(X,G',a')$ for some $G',a',X$. Since $\Gamma'_{\lambda+\Z a_1}$ and $\Gamma'_{\lambda+\Z a_2}$ are bounded, we have $a_1,a_2\in G'$. Then by Lemma \ref{cussubgroup}, $\Gamma'_{\lambda+a_1+\Z a_2}$ is bounded. Let $d_{a_2}\in D_0$ with $a_2(d_{a_2}),(\lambda+a_1)(d_{a_2}),(\lambda+a_1+a_2)(d_{a_2})\neq 0$. Then $\Gamma'_{\lambda+a_1+\Z a_2}$ is a bounded $\Vir[G_0,d_{a_2}]$-module. Therefore, $\dim \Gamma'_{\lambda+a_1}=\dim \Gamma'_{\lambda+a_1+a_2}$. This is a contradiction. Thus, $\Gamma_{\lambda+G_1+G_2}$ is bounded.
\end{proof}

\begin{lemma}\label{G=G0+b}
Let $\Gamma$ be a simple Harish-Chandra $W$-module. Then for any $\mu\in\supp(\Gamma)\setminus\{0\}$, there exists a unique maximal subgroup $G_{\mu}\leqslant G$ such that $\Gamma_{\mu+G_{\mu}}$ is bounded. Furthermore,
\begin{itemize}
\item[(1)]$G_{\mu_1}=G_{\mu_2}$ for any $\mu_1,\mu_2\in\supp(\Gamma)\setminus\{0\}$, which we denote by $G^{(0)}$;
\item[(2)]either $G^{(0)}=G$ or $G=G^{(0)}\oplus\Z b$ for some $b\in G\setminus\{0\}$.
\end{itemize}
\end{lemma}
\begin{proof}
The first statement follows directly from Lemma \ref{G1+G2}.

(1)Suppose there is $\mu_1,\mu_2\in\supp(\Gamma)\setminus\{0\}$ with $G_{\mu_1}\not\subset G_{\mu_2}$. Let $a_1\in G_{\mu_1}\setminus G_{\mu_2}$ and $d_{a_1}\in D$ with $\phi(a_1, d_{a_1})\neq 0$. Let $I=\{\mu_1, \mu_2\}, I_G=\{a_1\}, I_D=\{d_{a_1}\}$ and let $G_1\times D_1\subset G\times D$ be as in Lemma \ref{G1D1}. Then $\Gamma_{\mu_1+G_1}$ has a simple $W(\phi|_{G_1\times D_1})$-subquotient $\Gamma'$ with $\dim \Gamma_{\mu_i}=\dim \Gamma'_{\mu_i},i=1,2$.

If $\Gamma'$ is bounded, then $\Gamma'_{\mu_2+\Z a_1}$ is bounded. If $\Gamma'$ is not bounded, then we have $\Gamma'\cong L_{W(\phi|_{G_1\times D_1})}(X,G_0,a_0)$ for some $G_0,a_0,X$ by Theorem \ref{highranknotcus}. Since $\Gamma'_{\mu_1+\Z a_1}$ is bounded, we have $a_1\in G_0$. So $\Gamma'_{\mu_2+\Z a_1}$ is bounded by Lemma \ref{cussubgroup}.

Note that $\Gamma'_{\mu_2+\Z a_1}$ is a $\Vir[\Z a_1, d_{a_1}]$-subquotient of $\Gamma_{\mu_2+\Z a_1}$. By Lemma \ref{VirZa}, $\Gamma_{\mu_2+\Z a_1}$ is bounded. Then $a_1\in G_{\mu_2}$, which leads to a contradiction.

(2)Suppose $G^{(0)}\neq G$.

We claim that $G/G^{(0)}$ is torsion free.
If not, then there is an $a\in G\setminus G^{(0)}$ and a $k\in\N$ such that $ka\in G^{(0)}$. Let $\lambda \in\supp(\Gamma)\setminus\{0\}$, then $\Gamma_{\lambda+\Z a}$ is not bounded. Let $d_a\in D$ with $\phi(a,d_a)\neq 0$, then $\Gamma_{\lambda+\Z a}$ has a simple $\Vir[\Z a,d_a]$-subquotient $\Gamma'$ that is not bounded. Let $\mu\in\supp(\Gamma')\setminus\{0\}$, then $\Gamma'_{\mu+\Z ka}$ is not a bounded $\Vir[\Z a,d_a]$-module. This contradicts with the fact that $ka\in G^{(0)}$. Thus, $G/G^{(0)}$ is torsion free.


For any $a,b\in G\setminus G^{(0)}$, let $\lambda\in\supp(\Gamma)\setminus\{0\},G'=\Z a+\Z b$ and $d_0\in D$ with $\phi|_{G'\times\C d_0}$ non-degenerate. If $\Gamma_{\lambda+G'}$ is bounded, then $a,b\in G^{(0)}$. This is a contradiction. So $\Gamma_{\lambda+G'}$ is not bounded. It follows that $\Gamma_{\lambda+G'}$ has a simple $\Vir[G',d_0]$-subquotient $\Gamma'$ that is not bounded. By Theorem \ref{highranknotcus}, $\Gamma'\cong L_{\Vir[G',d_0]}(X',G'_0,a'_0)$ for some $G'_0,a'_0,X$. Follows from Lemma \ref{cussubgroup} and Lemma \ref{VirZa}, for any $c\in G'_0\setminus\{0\}$, $\Gamma_{\lambda+\Z c}$ is bounded. So $c\in G^{(0)}$ and $c=mb+na$ for some nonzero $m,n\in\Z$. Then, the images of $a$ and $b$ in $G/G^{(0)}$ are $\Z$-linearly dependent. Thus, $G/G^{(0)}$ is isomorphic to an additive subgroup of $\Q$.

Suppose that $G/G^{(0)}$ is not finitely generated.

If $G^{(0)}=0$, then $G$ is isomorphic to an additive subgroup of $\Q$ that is not finitely generated. So $\dim D=1$. By Theorem 1.1 and Remark 4.5 in \cite{M}, $\Gamma$ is a module of intermediate series (a weight module with each weight space one dimensional). Then $G=G^{(0)}$, which is a contradiction. So $G^{(0)}\neq 0$.

Since $G\neq G^{(0)}$, $\Gamma$ is not a bounded module. Then there is $\mu_1,\mu_2\in\supp(\Gamma)\setminus\{0\}$ with $\dim\Gamma_{\mu_1}\neq\dim\Gamma_{\mu_2}$. Then $c=\mu_1-\mu_2\in G\setminus G^{(0)}$. Since $G/G^{(0)}$ is isomorphic to an additive subgroup of $\Q$ that is not finitely generated, there is a $b_0\in G\setminus G^{(0)}$ such that $c\in k_0b_0+G^{(0)}$ for some $k_0\in\N$ with $k_0>\mathrm{max}\{\dim\Gamma_{\mu_1},\dim\Gamma_{\mu_2}\}$.

Let $a\in G^{(0)}\setminus\{0\},I=\{\mu_1,\mu_2\},I_G=\{a,b_0,c\}$, and let $G_1,D_1$ be as in Lemma \ref{G1D1}. Note that $G_1$ is free of finite rank, then there is a $b\in G$ such that for any $G_1\subset G^{(0)}+\Z b$. Clearly, $b\in G\setminus G^{(0)}$ and $c\in kb+G^{(0)}$ for some $k\in\N$ with $k\geqslant k_0>\mathrm{max}\{\dim\Gamma_{\mu_1},\dim\Gamma_{\mu_2}\}$. Add $b$ to $G_1$ and assume that $\phi|_{G_1\times D_1}$ is non-degenerate. By Lemma \ref{G1D1}, $\Gamma$ has a simple $W(\phi|_{G_1\times D_1})$-subquotient $\Gamma'$ such that $\dim\Gamma'_{\mu_i}=\dim\Gamma_{\mu_i},i=1,2$. $\dim\Gamma'_{\mu_1}\neq\dim\Gamma'_{\mu_2}$ implies that $\Gamma'$ is not a bounded module. By Theorem \ref{highranknotcus}, $\Gamma'\cong L_{W(\phi|_{G_1\times D_1})}(Y,H,x_0)$ for some $H,x_0,Y$. By Lemma \ref{VirZa}, $x_0\notin G^{(0)}$. Then $x_0\in rb+G^{(0)}$ for some $r\in\Z\setminus\{0\}$. Assume that $r>0$ (the case that $r<0$ could be dealt similarly). Then for any $\lambda\in\supp(\Gamma')$, the set $\{s\in\Z\ |\ \lambda+sb\in\supp(\Gamma')\}$ is upper-bounded.

Since $G/G^{(0)}$ is torsion-free, $a$ and $b$ are $\Z$-linearly independent. Then there is an $r\in\Z$ such that $\mu_1+ra,\mu_1+ra-kb\neq 0$ and $\mu_1+ra,b$ are $\Z$-linearly independent. Let $\lambda_1=\mu_1+ra,\lambda_2=\mu_1+r_a-k_b$, then $\lambda_i-\mu_i\in G^{(0)},i=1,2$. It follows that $\Gamma'_{\lambda_i+\Z(\mu_i-\lambda_i)},i=1,2,$ are bounded. Therefore, $\dim\Gamma'_{\lambda_i}=\dim\Gamma'_{\mu_i},i=1,2$. Let $d\in D_1$ such that $b(d)\neq 0$ and $\lambda_1(d)\notin\Q b(d)$. Then $\Gamma'_{\lambda_1+\Z b}$ has a simple $\Vir[\Z b,d]$-subquotient $\Gamma''$ with a highest weight in $\lambda_1+\Z_+b$. By (0.8) Corollary in \cite{RW}, any Verma module over $\Vir[\Z b,d]$ with the highest weight in $\lambda_1+\Z b$ is simple. So $\Gamma''$ is a Verma module and $\dim\Gamma''_{\lambda_2}\geqslant k$. It follows that $\dim\Gamma'_{\mu_2}>\dim\Gamma_{\mu_2}$. This is a contradiction.

Thus $G/G^{(0)}$ is finitely generated. So $G/G^{(0)}\cong\Z$ and $G=G^{(0)}\oplus\Z b$ for some $b\in G\setminus G^{(0)}$.
\end{proof}

\begin{theorem}\label{hwtype}
Suppose $\Gamma$ is a simple Harish-Chandra $W$-module that is not bounded. Then $\Gamma\cong L(X,G^{(0)},b)$ for some $b\in G\setminus\{0\}$ with $G=G^{(0)}\oplus\Z b$, and some simple bounded $W(\phi|_{G^{(0)}\times D})$-module $X$.
\end{theorem}
\begin{proof}
By the definition of $G^{(0)}$, we have $G\neq G^{(0)}$. Then by Lemma \ref{G=G0+b}(2), for some $b\in G\setminus\{0\}$. If $G^{(0)}=0$, then this theorem is true by Theorem \ref{highranknotcus}. Assume $G^{(0)}\neq 0$.

Let $\lambda\in\supp(\Gamma)\setminus\{0\}$. If $\lambda\in\Z b$, then $\lambda+a\notin\Z b$ for any $a\in G^{(0)}\setminus\{0\}$. Since $\Gamma_{\lambda+G^{(0)}}$ is a bounded $W(\phi|_{G^{(0)}\times D})$-module, $\dim\Gamma_{\lambda+a}=\dim\Gamma_{\lambda}\neq 0$ for any $a\in G^{(0)}$ with $\lambda+a\neq 0$. Without loss of generality, assume that $\lambda\notin\Z b$.

Let $d_0\in D$ with $b(d_0)\neq 0$ and $\lambda(d_0)\notin\Z b(d_0)$. Since $b\notin G^{(0)}$, $\Gamma_{\lambda+\Z}$ is not bounded. By Lemma \ref{VirZa}, any simple $\Vir[\Z b,d_0]$-subquotient of $\Gamma_{\lambda+\Z b}$ is not bounded, i.e. each simple $\Vir[\Z b,d_0]$-subquotient is a highest or lowest weight module.

Suppose that $\Gamma_{\lambda+\Z b}$ has a highest weight module and a lowest weight module as simple $\Vir[\Z b,d_0]$-subquotient. Let $\mu_1$ (resp. $v_{\mu_2}$) be a weight of a simple highest (resp. lowest) weight $\Vir[\Z b,d_0]$-subquotient. Let $I=\{\mu_1,\mu_2\},I_G=\{b\},I_D=\{d_0\}$, and let $G_1\times D_1\leqslant G\times D$ be as in Lemma \ref{G1D1}. Moreover, assume that $\phi|_{G_1\times D_1}$ is non-degenerate. Then $\Gamma$ has a simple $W(\phi_{G_1\times D_1})$-subquotient $\Gamma'$ such that $\dim\Gamma'_{\mu_i}=\dim\Gamma_{\mu_i},i=1,2$. Clearly, $\Gamma'$ is not bounded. By Theorem \ref{highranknotcus}, $\Gamma'\cong L_{W(\phi_{G_1\times D_1})}(Y,G',b')$ for some $G',b',Y$. It follows that $G'\leqslant G^{(0)}$ and $b'\in\pm b+G^{(0)}$. Then the support set of the $\Vir[\Z b,d_0]$-module $\Gamma'_{\mu_1+\Z b}$ is upper-bounded or lower-bounded. Note that any simple $\Vir[\Z b,d_0]$-subquotient of $\Gamma_{\mu_1+\Z b}$ with weight $\mu_1$ or $\mu_2$ is also a subquotient of $\Gamma'_{\mu_1+\Z b}$. Then the support set of the $\Vir[\Z b,d_0]$-module $\Gamma'_{\mu_1+\Z b}$ is not upper-bounded or lower-bounded. This is a contradiction.

Thus, any simple $\Vir[\Z b,d_0]$-subquotient of $\Gamma_{\lambda+\Z b}$ is a highest weight module, or any simple $\Vir[\Z b,d_0]$-subquotient of $\Gamma_{\lambda+\Z b}$ is a lowest weight module. Assume the former (the latter could be dealt similarly). Then $t^bd_0$ acts on $\Gamma_{\lambda+\Z b}$ locally nilpotently. By Lemma 2.3(2) in \cite{CLX}, there is a largest integer $n$ such that $\lambda+nb\in\supp(\Gamma)$, and for any $k\in\Z$ with $k>n$ we have $\lambda+kb\notin\supp(\Gamma)$.

For any $\mu\in\lambda+kb+G^{(0)}$ with $k>n$, $\Gamma_{\mu+G^{(0)}}$ is a bounded $W(\phi|_{G^{(0)}\times D})$-module. Then $\dim\Gamma_\mu=\dim\Gamma_{\lambda+kb}=0$. If $0\in\lambda+nb+\N b+G^{(0)}\cap\supp(\Gamma)$, then $\Gamma$ is a trivial $W$-module. This is a contradiction. So $\lambda+nb+\N b+G^{(0)}\cap\supp(\Gamma)=\varnothing$. By the PBW theorem, $\Gamma_{\lambda+nb+G^{(0)}}$ is a simple bounded $W(\phi|_{G^{(0)}\times D})$-module. Therefore, $\Gamma\cong L(\Gamma_{\lambda+nb+G^{(0)}},G^{(0)},b)$.
\end{proof}

Let $n\in\N$ and $f:G^n\rightarrow\C$ be a function. Define the linear map $\psi_f:\C[G^n]\rightarrow\C$ by
$$\psi_f(t^{\alpha})=f(\alpha),\alpha\in G^n.$$
Let $I_f$ be the maximal ideal of $\C[G^n]$ that is contained in $\Ker \psi_f$. Set
$$(G^n)^{\circledast}=\{f:G^n\rightarrow\C\ \text{is a function}\ |\ \dim\C[G^n]/I_f<\infty\}.$$

For $\alpha=(a_1,\dots,a_{n_1})\in G^{n_1},\beta=(b_1,\dots,b_{n_2})\in G^{n_2}$, write
$$(\alpha,\beta)=(a_1,\dots,a_{n_1},b_1,\dots,b_{n_2})\in G^{n_1+n_2}.$$
Let $f:G^{n_1} \rightarrow\C, g:G^{n_2}\rightarrow\C$ be functions. Define $(f,g):G^{n_1+n_2}\rightarrow\C$ by
$$(f, g)(\alpha,\beta)=f(\alpha)g(\beta),\alpha\in G^{n_1}, \beta\in G^{n_2}.$$

\begin{lemma}\label{finfun}
\begin{itemize}
\item[(1)]Suppose $f\in(G^{n_1})^{\circledast}, g\in(G^{n_2})^{\circledast}$. Then $(f, g)\in(G^{n_1+n_2})^{\circledast}$.
\item[(2)]$(G^n)^{\circledast}$ is an associative algebra with $(f_1+f_2)(\alpha)=f_1(\alpha)+f_2(\alpha), (f_1f_2)(\alpha)=f_1(\alpha)f_2(\alpha),f_1, f_2\in(G^n)^{\circledast}, \alpha\in G^n$.
\end{itemize}
\end{lemma}
\begin{proof}
(1)It is easy to see that the map
$$\C[G^{n_1+n_2}]\rightarrow \C[G^{n_1}]\otimes\C[G^{n_2}],t^{\alpha+\beta}\mapsto t^\alpha\otimes t^\beta,\alpha\in G^{n_1},\beta\in G^{n_2}$$
is a space isomorphism. Note that, under this isomorphism, $I_f\otimes\C[G^{n_2}]+\C[G^{n_1}]\otimes I_g\subset I_{(f,g)}$. Then
$$\dim\C[G^{n_1+n_2}]/I_{(f,g)}\leqslant\dim\C[G^{n_1}]/I_f+\dim\C[G^{n_2}]/I_g<\infty.$$
Thus, $(f,g)\in(G^{n_1+n_2})^\circledast$.

(2)It suffices to verify that $(G^n)^\circledast$ is closed under the given addition and multiplication, as well as the multiplication by a scalar. Let $f_1,f_2\in(G^n)^\circledast,k\in\C$. Then $I_{f_1}\subset I_{kf_1},I_{f_1}\cap I_{f_2}\subset I_{f_1+f_2}$. So
$$\dim\C[G^n]/I_{kf_1}\leqslant\dim\C[G^n]/I_{f_1}<\infty,\dim\C[G^n]/I_{f_1+f_2}\leqslant\dim\C[G^n]/(I_{f_1}\cap I_{f_2})<\infty.$$
Then $kf_1,f_1+f_2\in (G^n)^\circledast$.

Define the group monomorphism $i:G^n\rightarrow G^{2n}$ by $i(\alpha)=(\alpha, \alpha)$. Then $f_1f_2=(f_1, f_2)\circ i$. By (1), $(f_1, f_2)\in(G^{2n})^{\circledast}$. So $f_1f_2\in(G^n)^{\circledast}$. Hence, $(G^n)^{\circledast}$ is an associative algebra.
\end{proof}

\begin{lemma}\label{bbm1}
Suppose $n_1,n_2\in\N$. Then
\begin{itemize}
\item[(1)]$\mathbbm 1\in(G^{n_1})^\circledast$, where $\mathbbm 1:G^{n_1}\rightarrow\C$ is defined by $\mathbbm 1(\alpha)=1,\alpha\in G^{n_1}$;
\item[(2)]for any function $f:G^{n_1+n_2}\rightarrow\C$ such that
$$f(\alpha_1,\beta_1+k\beta_2)=f(\alpha_2,\beta_1)+kf(\alpha_2,\beta_2),\forall\alpha_1,\alpha_2\in G^{n_1},\beta_1,\beta_2\in G^{n_2},k\in\Z,$$
we have $f\in(G^{n_1+n_2})^\circledast$.
\end{itemize}
\end{lemma}
\begin{proof}
(1)Note that $\psi_{\mathbbm 1}$ is an algebra homomorphism and $\dim\C[G_0^{n_1}]/I_{\mathbbm 1}=1$. So $\mathbbm 1\in(G^{n_1})^\circledast$.

(2)Clearly, $f\in(G^{n_1+n_2})^\circledast$ if $f(G^{n_1+n_2})=0$. Suppose that $f(\alpha_1,\beta_1)\neq 0$. For any $\alpha,\alpha_0\in G^{n_1},\beta,\beta_0\in G^{n_2}$,
$$\big(f(\alpha_1,\beta_1)t^{(\alpha,\beta)}-f(\alpha,\beta)t^{\alpha_1,\beta_1}+f(\alpha_1,\beta-\beta_1)t^{(0,0)}\big)t^{(\alpha_0,\beta_0)}\in I_f.$$
Then
$$f(\alpha_1,\beta_1)t^{(\alpha,\beta)}-f(\alpha,\beta)t^{\alpha_1,\beta_1}+f(\alpha_1,\beta-\beta_1)t^{(0,0)}.$$
Therefore, $\dim\C[G^{n_1+n_2}]/I_f\leqslant 2$. Thus, $f\in(G^{n_1+n_2})^\circledast$.
\end{proof}

\begin{lemma}\label{GEPmodule}
Suppose $\dim D=r<\infty, G=G_0\oplus \Z a_0$ for some $G_0\leqslant G, a_0\in G\setminus\{0\}$ and $G=G_0\oplus\Z a_0$.
Let $X$ be a simple bounded $W(\phi|_{G_0\times D})$-module and $\lambda\in\supp(X)$. Then there is a basis $\{d_0, ..., d_{r-1}\}$ of $D$, a finite index set $K$, a spanning set $\{v_{\lambda+a}^{(i)}\in X_{\lambda+a}\ |\ i\in K\}$ of $X_{\lambda+a}$ for all $a\in G_0$, and a family of functions $\{f^{(i, j, l)}\in(G_0^2)^{\circledast}\ |\ i\in\{0, ..., r-1\}, j, l\in K\}$ such that
$$t^{\alpha}d_i\cdot v_{\lambda+\beta}^{(j)}=\Sigma_{l\in K}f^{(i, j, l)}(\alpha, \beta)v_{\lambda+\alpha+\beta}^{(l)},\forall\alpha, \beta\in G_0.$$
\end{lemma}
\begin{proof}
By Theorem \ref{bounded1}, $X$ is a simple $W(\phi|_{G_0\times D})$-quotient of $\Gamma_{W(\phi|_{G_0\times D})}(V, \lambda)$ for some finite dimensional simple $G_0\otimes_{\Z}D$-module $V$. It suffices to prove that $\Gamma_{W(\phi|_{G_0\times D})}(V, \lambda)$ has the property as stated.

Let $d_0, ..., d_{r-1}$ be a basis of $D$, and $v_1, ..., v_k$ be a basis of $V$. Then
$$(\Gamma_{W(\phi|_{G_0\times D})}(V, \lambda))_{\lambda+\alpha}=\span\{t^{\alpha}\otimes v_i\ |\ i=1, ..., k\}$$
and
$$t^{\alpha}d_i\cdot t^{\beta}\otimes v_j=t^{\alpha+\beta}\otimes((\lambda+\beta)(d_i)+\alpha\otimes d_i)v_j$$
for all $\alpha, \beta\in G_0$.
Let $f_1^{(i,j,l)}(\alpha,\beta)=\delta_{j, l}\lambda(d_i), f_2^{(i, j, l)}(\alpha, \beta)=\delta_{j, l}\beta(d_i)$, and let $f_3^{(i, j, l)}$ be a function such that $(\alpha\otimes d_i)\cdot v_j=\Sigma_{l=1}^kf_3^{(i, j, l)}(\alpha, \beta)v_l$. By Lemma \ref{bbm1}, $f_1,f_2,f_3\in(G_0^2)^\circledast$.
It follows that $f^{(i, j, l)}=f_1^{(i, j, l)}+f_2^{(i, j, l)}+f_3^{(i, j, l)}\in(G_0^2)^{\circledast}$. Then $\Gamma_{W(\phi|_{G_0\times D})}(V,\lambda)$ has the property as stated.
\end{proof}

\begin{theorem}\label{HC}
Suppose $\dim D=r<\infty, G=G_0\oplus\Z a_0$ for some $G_0\leqslant G,a_0\in G\setminus\{0\}$. Let $X$ be a simple bounded $W(\phi|_{G_0\times D})$-module. Then $\Gamma=L(X,G_0,a_0)$ is a Harish-Chandra module.
\end{theorem}
\begin{proof}
Let $\lambda\in\supp(X)$ and use the notations in Lemma \ref{GEPmodule}. For any $i_1, i_2\in\Z, j_1, j_2\in\{0, ..., r-1\}$, define a function $g^{(i_1,i_2,j_1,j_2)}:G_0^2\rightarrow\C$ by
$$g^{(i_1,i_2,j_1,j_2)}(\alpha_1,\alpha_2)=(i_2a_0+\alpha_2)(d_{j_1}),\alpha_1,\alpha_2\in G_0.$$
Then
\begin{equation*}\begin{split}
[t^{i_1a_0+\alpha_1}d_{j_1},t^{i_2a_0+\alpha_2}d_{j_2}]=&g^{(i_1,i_2,j_1,j_2)}(\alpha_1,\alpha_2)t^{(i_1+i_2)a_0+\alpha_1+\alpha_2}d_{j_2}\\
&-g^{(i_2,i_1,j_2,j_1)}(\alpha_2,\alpha_1)t^{(i_1+i_2)a_0+\alpha_1+\alpha_2}d_{j_1}.
\end{split}\end{equation*}
By Lemma \ref{bbm1}, $g^{(i_1,i_2,j_1,j_2)}\in (G_0^2)^\circledast$.


Claim. For any $i_1, ..., i_s\in\N, m_1, ..., m_s\in\{0, ..., r-1\}, k\in K, \alpha\in G_0$,
$$\span\{t^{-i_1a_0+\beta_1}d_{m_1}\cdot\dots\cdot t^{-i_sa_0+\beta_s}d_{m_s}\cdot v_{\lambda+\alpha-\beta_1-...-\beta_s}^{(k)}\in\Gamma\ |\ \beta_1, ..., \beta_s\in G_0\}$$
is finite dimensional.

Let $\beta=(\beta_1,\dots,\beta_s)$ and $x(\beta)=t^{-i_1a_0+\beta_1}d_{m_1}\cdot\dots\cdot t^{-i_sa_0+\beta_s}d_{m_s}\cdot v_{\lambda+\alpha-\beta_1-...-\beta_s}^{(k)}$. Define a linear map $\theta_{\alpha}:\C[G^s]\rightarrow\Gamma_{-(i_1+...+i_s)a_0+\alpha}$ by $\theta_\alpha(t^{\beta})=x(\beta)$.

For any $j_1,...,j_t\in\N$ with $j_1+...+j_t=i_1+...+i_s,n_1,...,n_t\in\{0, ..., r-1\},\gamma_1,...,\gamma_t\in G_0$, let $\gamma=(\gamma_1,\dots,\gamma_t)$ and $h_\omega:G_0^{1+s+t}\rightarrow\C$ be functions such that
$$t^{j_1a_0+\gamma_1}d_{n_1}\cdot\dots\cdot t^{j_ta_0+\gamma_t}d_{n_t}\cdot x(\beta)=\Sigma_{l\in K}h_{\omega}(\alpha, \beta, \gamma)v_{\lambda+\alpha+\gamma_1+...+\gamma_t}^{(l)},$$
where $\omega=(j_1,...,j_t,n_1,...,n_t,l)$. By Lemma \ref{finfun}, $h_\omega\in(G_0^{1+s+t})^{\circledast}$.

Define $h_{\omega,\alpha,\gamma}:G_0^s\rightarrow\C$ by $h_{\omega,\alpha,\gamma}(\beta)=h_\omega(\alpha,\beta,\gamma)$. Since $h_\omega\in(G_0^{1+s+t})^{\circledast}$, $h_{\omega,\alpha,\gamma}\in(G_0^s)^\circledast$. Moreover, it is easy to see that $I_{h_{\omega,\alpha,\gamma}}=I_{h_{\omega,\alpha,\gamma'}}$ for any $\gamma,\gamma'\in G_0^t$. Let $I=\sum_\omega I_{h_{\omega,\alpha,\gamma}}$, where $\omega$ sums over all $(j_1,...,j_t,n_1,...,n_t,l)$ with $i_1+\dots+i_s=j_1+\dots+j_t$. This is a finite sum. So $\dim\C[G_0^s]/I<\infty$.

For any $x\in I,\gamma\in G_0^t$, $t^{j_1a_0+\gamma_1}d_{n_1}\cdot\dots\cdot t^{j_ta_0+\gamma_t}d_{n_t}\cdot\theta_\alpha(x)=0$. By the PBW theorem, $\theta_\alpha(x)=0$. Thus, $I\subset\Ker\theta_\alpha$. It follows that $\dim \theta_{\alpha}(\C[G^s])\leqslant\dim\C[G^s]/I<\infty$. Thus, the Claim is proved.

For any $i\in\N, \alpha\in G_0$,
$$\Gamma_{-ia_0+\alpha}=\span\{t^{-i_1a_0+\beta_1}d_{m_1}\cdot...\cdot t^{i_sa_0+\beta_s}d_{m_s}\cdot v_{\alpha-\beta_1-...-\beta_s}^{(k)}|$$
$$i_1+...+i_s=i, i_1, ..., i_s, s\in\N, \beta_1, ..., \beta_s\in G_0, k\in K, m_1, ..., m_s\in\{0, ..., r-1\}\}.$$
By the Claim, $\Gamma_{-ia_0+\alpha}$ is finite dimensional. Therefore, $\Gamma=L(X,G_0,a_0)$ is a Harish-Chandra module.
\end{proof}

\begin{theorem}\label{infweightspace}
Suppose $\dim D=\infty$ and $G=G_0\oplus\Z a_0$ for some $G_0\leqslant G, a_0\in G\setminus\{0\}$. Let $X$ be a nontrivial simple bounded $W(\phi|_{G_0\times D})$-module, $\Gamma=L(X,G_0,a_0)$. Then $\Gamma$ has a weight space that is not finite dimensional.
\end{theorem}
\begin{proof}
It is easy to see that $\dim\Ker_2\phi|_{G_0\times D}\leqslant 1$. Then by Theorem \ref{bounded2}, $X$ is isomorphic to
\begin{itemize}
\item[(1)]$\Gamma_{W(\phi|_{G_0\times D})}(V^c,\sigma)$, where $c\in\C,\sigma\in D^*$ with $\sigma\notin G$ or $c\notin\{0,1\}$;
\item[(2)]$\Gamma_{W(\phi|_{G_0\times D})}(\sigma,1)$ or $\Gamma_{W(\phi|_{G_0\times D})}(\sigma,\infty)$, where $\sigma\in G$.
\end{itemize}

Let $c=0,1$ if $X\cong\Gamma_{W(\phi|_{G_0\times D})}(\sigma,1),\Gamma_{W(\phi|_{G_0\times D})}(\sigma,\infty)$ respectively. Then $X$ contains linearly independent vectors $\{v_{\sigma+a}\ |\ a\in G_0,a+\sigma\neq 0\}$ such that
$$t^bd\cdot v_{\sigma+a}=(\sigma+a+cb)(d)v_{\sigma+a+b},\forall a,b\in G_0,d\in D\ \text{with}\ \sigma+a,\sigma+a+b\neq 0.$$

Let $d'\in D\setminus\{0\}$ with $a_0(d')=0$. Replace $\sigma$ by $\sigma+a$ for some $a\in G_0$ if necessary and assume that $\sigma(d')\neq 0$ and $\sigma\neq a_0$.

For any $\beta\in G_0$ and $d_1',d_2'\in D$ with $\sigma\neq\beta$, we have
\begin{equation*}\begin{split}
&t^{a_0}d_1'\cdot t^{-a_0+\beta}d_2'\cdot v_{\sigma-\beta}\\
=&\big((-a_0+\beta)(d_1')t^{\beta}d_2'-a_0(d_2')t^{\beta}d_1'\big)\cdot v_{\sigma-\beta}\\
=&\big((-a_0+\beta)(d_1')(\sigma+(c-1)\beta)(d_2')-a_0(d_2')(\sigma+(c-1)\beta)(d_1')\big)v_{\sigma}.
\end{split}\end{equation*}
Note that $\dim(\Ker a_0\cap\Ker\sigma)=\infty$. For any $s\in\N$, let $\beta_1,\dots,\beta_s\in G_0$ such that $\beta_i|_{\Ker a_0\cap\Ker\sigma},i=1,\dots,s$, are $\C$-linearly independent. Let $d_i\in \Ker a_0\cap\Ker\sigma,i=1,\dots,s$ such that $\beta_i(d_j)=\delta_{i,j}$.

Case 1. $c=1$.

For any $i,j\in\{1,\dots,s\}$, $t^{a_0}d_i\cdot t^{-a_0+\beta_j}d'\cdot v_{\sigma-\beta_j}=\delta_{i,j}\sigma(d')v_{\sigma}$. If $\Sigma_{j=1}^sc_jt^{-a_0+\beta_j}d'\cdot v_{\sigma-\beta_j}=0$, then $c_j\sigma(d')v_{\sigma}=0$ for all $j$. So $c_j=0, j=1, ..., s$. Thus, $\{t^{-a_0+\beta_j}d'\cdot v_{\sigma-\beta_j}\ |\ j=1, ..., s\}$ is linearly independent.

Case 2. $c\neq 1$.

For any $i,j\in\{1,\dots,s\}$, $t^{a_0}d_i\cdot t^{-a_0+\beta_j}(d_1+...+d_s)\cdot v_{\sigma-\beta_j}=\delta_{i, j}(c-1)v_{\sigma}$. Similarly, $\{t^{-a_0+\beta_j}(d_1+...+d_s)\cdot v_{\sigma-\beta}\ |\ j=1, ..., s\}$ is linearly independent.

Since $s$ is arbitrary, $\dim \Gamma_{\sigma-a_0}=\infty$.
\end{proof}

Here we give the main theorem of this paper.
\begin{theorem}\label{main}
Suppose $\Gamma$ is a simple Harish-Chandra $W$-module. If $\dim D<\infty$, then
\begin{itemize}
\item[(1)]$\Gamma$ is a bounded module if and only if one of the following cases happens:
\begin{itemize}
\item[(1)]$\Gamma\cong\Gamma(V,\sigma)$, where $V$ is a finite dimensional simple $L$-module, $\sigma\in D^*$ with $V\ncong\wedge^i\bar D^*,i=0,\dots,r$.
\item[(2)]$\Gamma\cong\Gamma(\sigma,k),k\in\{1,\dots,r\}$, where $\sigma\in D^*$.
\item[(3)]$\Gamma$ is the one dimensional trivial $W$-module.
\end{itemize}
\item[(2)]$\Gamma$ is not bounded if and only if $\Gamma\cong L(X,G^{(0)},b)$ for some $b\in G\setminus\{0\}$ with $G=G^{(0)}\oplus\Z b$, and some simple bounded $W(\phi|_{G^{(0)}\times D})$-module $X$.
\end{itemize}
If $\dim D=\infty$, then $\Gamma$ is a bounded module. Moreover, one of the following cases happens:
\begin{itemize}
\item[(1)]$\Gamma\cong\Gamma(V^c,\sigma)$, where $c\in\C,\sigma\in D^*$ with $\sigma\notin G$ or $c\notin\{0,1\}$.
\item[(2)]$\Gamma\cong\Gamma(\sigma,1)$ or $\Gamma(\sigma,\infty)$, where $\sigma\in G$.
\item[(3)]$\Gamma$ is the one dimensional trivial module.
\end{itemize}
\end{theorem}
\begin{proof}
Follows from Theorem \ref{bounded1}, Theorem \ref{bounded2}, Theorem \ref{hwtype}, Theorem \ref{HC} and Theorem \ref{infweightspace}.
\end{proof}

Here we give some examples of simple Harish-Chandra modules over some simple generalized Witt algebras.

\textbf{Example 1}. Suppose $G=\Z^\infty$ with a standard basis $\{e_i\ |\ i\in\N\}$, and $D$ is a infinite dimensional vector space with basis $\{d_i\ |\ i\in\N\}$, such that $\phi(e_i,d_j)=\delta_{i,j},i,j\in\N$. In this case, for any $c\in\C$, $V^c=\C v$ is the one dimensional $L$-module with $e_i\otimes d_j\cdot v=c\delta_{i,j}v$. Then for any $\sigma\in D^*$, $\Gamma(V^c,\sigma)$ has a basis $\{v_b\ |\ b\in G\}$ with
$$t^ad_i\cdot v_b=(a_i+b_i+\sigma(d_i))v_{a+b},a,b\in G.$$
If $\sigma\in G$ and $c=0$, $\Gamma(\sigma,1)=\Gamma(V^0,\sigma)/\C v_{-\sigma}$. If $\sigma\in G$ and $c=1$, $\Gamma(\sigma,\infty)=\span\{v_a\ |\ a\in G,a\neq-\sigma\}$. By Theorem \ref{main}, $\Gamma$ is a simple Harish-Chandra $W$-module if and only if one of the following cases happens:
\begin{itemize}
\item[(1)]$\Gamma\cong\Gamma(V^c,\sigma)$, where $c\in\C,\sigma\in D^*$ with $\sigma\notin G$ or $c\notin\{0,1\}$.
\item[(2)]$\Gamma\cong\Gamma(\sigma,1)$ or $\Gamma(\sigma,\infty)$, where $\sigma\in G$.
\item[(3)]$\Gamma$ is the one dimensional trivial module.
\end{itemize}

\textbf{Example 2}. Let $G=\Z e_1+\Z e_2+\Z e_3$ and $D=\span\{d_1,d_2+\sqrt 2d_3\}$ with $\phi(e_i,d_j)=\delta_{i,j},i,j=1,2,3$. Then $\phi$ is non-degenerate and $W(\phi)$ is a simple generalized Witt algebra. Note that the Lie algebra $G\otimes_\Z D/(\Ker_1+\Ker_2)\cong\gl_2$. Let $V$ be a $\gl_2$-module and $\sigma\in D^*$ with $\sigma(d_1)=\sigma_1,\sigma(d_2+\sqrt 2 d_3)=\sigma_2$, the tensor module $\Gamma(V,\sigma)$ is given by
$$t^\alpha d_1\cdot t^\gamma\otimes v=t^{\alpha+\gamma}\otimes(\gamma_1+\sigma_1+\alpha_1E_{1,1}+(\alpha_2+\sqrt 2\alpha_3)E_{2,1})v,$$
$$t^\alpha(d_2+\sqrt 2 d_3)\cdot t^\gamma\otimes v=t^{\alpha+\gamma}\otimes(\gamma_2+\sigma_2+\sqrt 2(\gamma_3+\sigma_3)+\alpha_1E_{1,2}+(\alpha_2+\sqrt 2\alpha_3)E_{2,2})v,$$
$$\alpha=\alpha_1e_1+\alpha_2e_2+\alpha_3e_3\in G,\gamma=\gamma_1e_1+\gamma_2e_2+\gamma_3e_3\in G,v\in V.$$
Let $G_0\leqslant G$ and $0\neq a_0\in G_0$ with $G_0\oplus\Z a_0=G$. Then $W(\phi|_{G_0\times D})$ is isomorphic to the extended higher rank Virasoro algebra $\Vir[G_0,d_2+\sqrt 2 d_3]\ltimes\C[G_0]$ if $G_0=\Z e_2+\Z e_3$. Otherwise, $W(\phi|_{G_0\times D})$ is isomorphic to the Witt algebra $W_2$. And any non-bounded simple Harish-Chandra $W$-module is isomorphic to $\Gamma\cong L(X,G_0,a_0)$ for some $a_0\in G\setminus\{0\}$ with $G=G_0\oplus\Z a_0$, and some simple bounded $W(\phi|_{G_0\times D})$-module $X$.

\
\noindent{\bf Acknowledgement:}  This paper is partially supported by NSF of China (Grant 12271383, 11971440).


R. L\"u: Department of Mathematics, Soochow University, Suzhou, P. R. China.

Email address: rlu@suda.edu.cn

Y. Xue: School of Sciences, Nantong University, Nantong, Jiangsu, 226019, P. R. China.

Email address: yxue@ntu.edu.cn


\begin{thebibliography}{99999}



\bibitem{BF1} Y. Billig, V. Futorny, Classification of irreducible representations of Lie algebra of vector fields on a torus, J. Reine Angew. Math., 720 (2016), 199-216.
\bibitem{BF2} Y. Billig, V. Futorny, Classification of simple cuspidal modules for solenoidal Lie algebras, Israel J. Math., 222 (2017), no. 1, 109-123.
\bibitem{BFIK}  Y. Billig, V. Futorny, K. Iohara, I. Kashuba, Classification of simple strong Harish-Chandra $W(m,n)$-modules, arXiv:2006.05618.
\bibitem{BZ} Y. Billig, K. Zhao, Weight modules over exp-polynomial Lie algebras, J. Pure Appl.
Algebra, 191(2004), 23-42.
\bibitem{C}  E. Cartan, Les groupes de transformations continus, infinis, simples (French), Ann. Sci. Éc. Norm. Supér. 26 (1909), 93-161.
\bibitem{CLX}  Y. Cai, R. L\"{u}, Y. Xue, Classification of simple strong Harish-Chandra modules over the Lie superalgebra of vector fields on $\C^{m|n}$, arXiv:2106.04801.
\bibitem{DZ} D. Djokovic, K. Zhao, Derivations, isomorphisms, and second cohomology of generalized Witt algebras, Trans. Amer. Math. Soc. 350 (1998), no. 2, 643-664.
\bibitem{E1} S. Eswara Rao, Irreducible representations of the Lie-algebra of the diffeomorphisms of a $d$-dimensional torus, J. Algebra, {\bf 182}  (1996),  no. 2, 401-421.
\bibitem{E2} S. Eswara Rao, Partial classification of modules for Lie algebra of diffeomorphisms of d-dimensional torus, J. Math. Phys., 45 (8), (2004) 3322-3333.


\bibitem{GLLZ}  X. Guo, G. Liu, R. Lü, K. Zhao, Simple Witt modules that are finitely generated over the Cartan subalgebra, Mosc. Math. J. 20 (1) (2020) 43–65.
\bibitem{GLZ} X. Guo, G. Liu, K. Zhao, Irreducible Harish-Chandra modules over extended Witt algebras, Ark. Mat., 52 (2014), 99-112.
\bibitem{GLZ1}  X. Guo, R. L\"u, K. Zhao, Classification of irreducible Harish-Chandra modules over generalized Virasoro algebras, Proc. Edinb. Math. Soc. (2) 55 (2012), no. 3, 697–709.

\bibitem{H} J. E. Humphreys, ``Introduction to Lie Algebras and Representations Theory," Springer-Verlag, Berlin Heidelberg, New York, 1972.
\bibitem{Kac}  V.G. Kac, Some problems of infinite-dimensional Lie algebras and their representations, in: Lie Algebras and Related Topics, in: Lecture Notes in Math., vol. 933, Springer-Verlag, Berlin, 1982, pp. 117–126.
\bibitem{K}  I. Kaplansky, Seminar on simple Lie algebras, Bull. Am. Math. Soc. 60 (1954), 470-471.
\bibitem{K1}  N. Kawamoto, Generalizations of Witt algebras over a field of characteristic zero, Hiroshima Math. J. 16 (1986), 417-426.
\bibitem{LZ2}R. Lu, K. Zhao, Classification of irreducible weight modules over higher rank Virasoro algebras, Adv. Math. 201(2006), no. 2, 630-656.
\bibitem{LZ1}R. Lu, K. Zhao, Classification of irreducible weight modules over the twisted Heisenberg-Virasoro algebra, Commun. Contemp. Math. 12(2010), 183-205.

\bibitem{Mat}  O. Mathieu, Classification of Harish-Chandra modules over the Virasoro Lie algebras, Invent. Math. 107 (1992) 225–234.
\bibitem{M} V. Mazorchuk, Classification of simple Harish-Chandra modules over $\Q$-Virasoro algebra, Math. Nachr. 209(2000), 171-177.
\bibitem{MZ} V. Mazorchuk, K. Zhao, Supports of weight modules over Witt algebras, Proc. Roy. Soc. Edinburgh Sect., A 141(2011), no. 1, 155-170.

\bibitem{RW}  A. Rocha-Caridi, N. Wallach, Characters of irreducible representations of the Virasoro algebra, Math. Z. 185 (1984), no. 1, 1–21.
\bibitem{Su}  Y. Su, Classification of indecomposable $sl_2(\C)$ modules and a conjecture of Kac on irreducible modules over the Virasoro algebra, J. Algebra 161 (1993) 33–46.
\bibitem{Su1}  Y. Su, Simple modules over the high rank Virasoro algebras, Commun. Algebra 29 (2001) 2067–2080.
\bibitem{XL}  Y. Xue, R. L\"u, Simple weight modules with finite-dimensional weight spaces over Witt superalgebras, J. Algebra 574 (2021), 92–116.

\bibitem{Z2} K. Zhao, Weight modules over generalized Witt algebras with $1$-dimensional weight spaces, Forum Math. 16(2004), 725-748.







\end{thebibliography}
\end{document}